\documentclass[graybox]{svmult}

\usepackage{mathptmx}
\usepackage{helvet}
\usepackage{courier}
\usepackage{type1cm}
\usepackage{makeidx}
\usepackage{graphicx}
\usepackage{multicol}
\usepackage[bottom]{footmisc}
\usepackage{amsmath}
\usepackage{amssymb}
\usepackage{algorithm}
\usepackage{algorithmic}
\usepackage{xspace}
\usepackage{comment}
\usepackage{helvet}
\usepackage{courier}
\usepackage{type1cm}
\usepackage{makeidx}
\usepackage{graphicx}
\usepackage{multicol}
\usepackage[bottom]{footmisc}
\usepackage{framed}
\usepackage{color}

\definecolor{oma}{RGB}{211,230,235}
\usepackage{subfigure}
\usepackage{tikz}
\usepackage{verbatim}
\usepackage{url}

\smartqed

\DeclareMathOperator{\Int}{int}

\DeclareMathOperator{\co}{cone}
\DeclareMathOperator{\dist}{dist}

\DeclareMathOperator{\Ind}{Ind}
\DeclareMathOperator{\Dom}{dom}

\newcommand{\R}{\mathbb{R}}

\newcommand{\x}{{\vec x}}
\newcommand{\y}{{\vec y}}
\newcommand{\z}{{\vec z}}
\newcommand{\f}{{\vec f}}

\newcommand{\ba}{{\vec a}}
\newcommand{\bb}{{\vec b}}
\newcommand{\bc}{{\vec c}}

\newcommand{\be}{{\vec e}}
\newcommand{\bg}{{\vec g}}

\newcommand{\bp}{{\vec p}}

\newcommand{\bu}{{\vec u}}
\newcommand{\bv}{{\vec v}}
\newcommand{\bw}{{\vec w}}

\newcommand{\iprod}[2]{\langle#1,#2\rangle}

\makeindex

\newcommand{\sdot}[2]{{#1 #2}}

\makeindex

\newcommand{\e}		{\varepsilon}
\newcommand{\la}	{\langle}
\newcommand{\ra}	{\rangle}

\newcommand{\ee}	{{\epsilon}}
\newcommand{\XS}	{{X_\star}}
\newcommand{\cdp}	{{\partial_c}}

\newcommand{\xs}	{{\x^\star}}
\newcommand{\xk}	{{\x^k}}

\newcommand{\EndProof}	{\rule{2ex}{2ex}\xspace}
\newcommand{\bleq}[2]	{\begin{equation}\label{#1}{#2}\end{equation}}

\def\ag#1{{\color{black}#1}} 
\def\pd#1{{\color{black}#1}} 

\begin{document}

\title*{Advances in Low-Memory Subgradient Optimization}
\titlerunning{Advances Low-Memory Subgradient Optimization}
\author{Pavel E. Dvurechensky, Alexander V. Gasnikov, Evgeni A. Nurminski and Fedor S. Stonyakin}
\authorrunning{P.E. Dvurechensky, A.V. Gasnikov, E.A. Nurminski and Fedor S. Stonyakin}
\institute{Pavel E. Dvurechensky \at Weierstrass Institute for Applied Analysis and Stochastic, Mohrenstr. 39, Berlin, 10117, Germany and Institute for Information Transmission Problems RAS, Bolshoy Karetny per. 19, build.1, Moscow, 127051, Russia \email{pavel.dvurechensky@wias-berlin.de}
\and Alexander V. Gasnikov \at Moscow Institute of Physics and Technology, 9 Institutskiy per., Dolgoprudny, Moscow Region, 141701, Russia \email{gasnikov@yandex.ru}
\and Evgeni A. Nurminski \at Far Eastern Federal University, Russky ostrov, Vladivostok, 690000, Russia \email{nurminskiy.ea@dvfu.ru}
\and Fedor S. Stonyakin \at V.I. Vernadsky Crimean Federal University, 4 V. Vernadsky Ave, Simferopol, 295007 and Moscow Institute of Physics and Technology, 9 Institutskiy per., Dolgoprudny, Moscow Region, 141701 \email{fedyor@mail.ru}
}
\maketitle

\begin{abstract}
\makeatletter{}
One of the main goals in the development of non-smooth optimization is to cope
with high dimensional problems by decomposition, duality or Lagrangian relaxation which greatly
reduces the number of variables at the cost of worsening differentiability of objective or constraints.
Small or medium dimensionality of resulting non-smooth problems allows
to use bundle-type algorithms to achieve higher rates of convergence and obtain higher accuracy,
which of course came at the cost of additional memory requirements, typically of the order of \(n^2\),
where \(n\) is the number of variables of non-smooth problem.
However with the rapid development of more and more sophisticated models in industry, economy,
finance, et all such memory requirements are becoming too hard to satisfy.
It raised the interest in subgradient-based low-memory algorithms and later developments in this area
significantly improved over their early variants still preserving \(O(n)\) memory requirements. 
To review these developments this chapter is devoted to the black-box subgradient algorithms
with the minimal requirements for the storage of auxiliary results, which are necessary to execute these algorithms.
To provide historical perspective this survey starts with the original result of N.Z. Shor
which opened this field with the application to the classical transportation problem.
The theoretical complexity bounds for smooth and non-smooth convex and quasi-convex optimization problems
are briefly exposed in what follows to introduce to the relevant fundamentals of non-smooth optimization.
Special attention in this section is given to the adaptive step-size policy which
aims to attain lowest complexity bounds. 
Unfortunately the non-differentiability  of objective function in convex optimization
essentially slows down the theoretical low bounds for the rate of convergence in subgradient optimization
compared to the smooth case but there are different modern techniques that allow
to solve non-smooth convex optimization problems faster then dictate lower complexity bounds.
In this work the particular attention is given to Nesterov smoothing technique, Nesterov Universal approach,
and Legendre (saddle point) representation approach.
The new results on Universal Mirror Prox algorithms represent the original parts of the survey.
To demonstrate application of non-smooth convex optimization algorithms for solution of huge-scale extremal problems we consider
convex optimization problems with non-smooth functional constraints and propose two adaptive Mirror Descent methods.
The first method is of primal-dual variety and proved to be optimal in terms of lower oracle bounds for the
class of Lipschitz-continuous convex objective and constraints.
The advantages of application of this method to sparse Truss Topology Design problem are discussed in certain details.
The second method can be applied for solution of convex and quasi-convex optimization problems and
is optimal in a sense of complexity bounds.
The conclusion part of the survey contains the important references that characterize recent developments of non-smooth convex
optimization.
\end{abstract}
\section*{Introduction}
\label{intro}
We consider a finite-dimensional non-differentiable convex optimization problem (COP)
\bleq{copt}{ \min_{x\in E} f(x) = f_\star = f(\xs), \xs \in X_\star\,,}
where
\(E\) denotes a finite-dimensional space of primal variables and \(f:E\to\R\) is a finite convex function,
not necessarily differentiable.
For a given point \(\x\) the subgradient oracul returns value of objective function at that point \(f(\x)\)
and subgradient \(g \in \partial f(\x)\).
We do not make any assumption about the choice
of \(g\) from \( \partial f(\x)\).
As we are interested in computational issues related to solving (\ref{copt})
mainly we assume that this problem is solvable and has nonempty and bounded set of solutions \(X_\star\).

This problem enjoys a considerable popularity due to its important theoretical properties and numerous
applications in large-scale structured optimization, discrete optimization,
exact penalization in constrained optimization, and others.
Non-smooth optimization theory made it possible to solve in an efficient way classical discrete min-max problems
\cite{ddm2002}, \(l_1\)-approximation and others, at the same time opening new approaches
in bi-level, monotropic programming, two-stage stochastic optimization,
to name a few.

As a major steps in:the development
of different algorithmic ideas we can start with the subgradient algorithm due
to Shor (see \cite{Shor2012} for the overview and references to earliest works).
\section{Example Application: Transportation Problem and The First Subgradient Algorithm}
\label{exap:trans}
From utilitarian point of view the development of non-smooth (convex) optimization
started with the classical transportation problem
\bleq{trans-prob}{
\begin{array}{c}
 \min\ \sum_{i=1}^m \sum_{j=1}^n c_{ij} x_{ij} \\
\sum_{i=1}^m x_{ij} = a_j,\ j=1,2,\dots,n; \\
\sum_{j=1}^n x_{ij} = b_i,\ i=1,2,\dots,m \\
x_{ij} \geq 0, i=1,2,\dots,m; j=1,2,\dots,n
\end{array}
}
which is widely used in many applications.

By dualizing this problem with respect to balancing constrains we can
convert (\ref{trans-prob}) into dual problem of the kind
\bleq{dual-trans-prob}{ \max\ \Phi(\bu, \bv)}
where \(\bu = (u_i, i=1,2,\dots,m); \bv = (v_j, j=1,2,\dots,n) \)
are dual variables associated with the balancing constraints in (\ref{trans-prob})
and \( \Phi(\bu, \bv) \) is the dual function defined as
\bleq{fun-trans-prob}{ \Phi(\bu, \bv) = \inf_{\x \geq 0 } L(\x, \bu, \bv)}
and \( L(\x, \bu, \bv) \) is the Lagrange function of the problem:
\[
L(\x, \bu, \bv) = \sum_{i=1}^m \sum_{j=1}^n c_{ij} x_{ij} + \\
\sum_{j=1}^n u_j (\sum_{i=1}^m x_{ij} - a_j ) + \sum_{i=1}^m v_i (\sum_{j=1}^n x_{ij} - b_i ).
\]
By rearranging terms in this expression we can obtain the following expression for the dual function
\bleq{dualfun}{
\begin{array}{c}
\Phi(\bu, \bv) = -m\sum_{j=1}^n u_j a_j - n\sum_{i=1}^m v_i b_i +
\sum_{i=1}^m \sum_{j=1}^n \inf_{\x \geq 0} x_{ij} \{ c_{ij} + u_j + v_i \} = \\
- m\sum_{j=1}^n u_j a_j - n\sum_{i=1}^m v_i b_i - \Ind_D(\bu, \bv),
\end{array}
}
where
\bleq{indi}{
\Ind_D(\bu, \bv) = \left\{
\begin{array}{ll}
0 & \mbox{ when } c_{ij} + u_i + v_j \geq 0;
\\
\infty & \mbox{ otherwise.}
\end{array}
\right.
}
is the indicator function of the set \(D = \{ \bu, \bv: c_{ij} + u_j + v_i \geq 0,
i=1,2,\dots,m; j=1,2,\dots,n \} \) which is the feasible set of the dual problem.

Of course, by explicitely writing feasibility constraints for (\ref{dual-trans-prob})
we obtain the linear dual transportation problem
with a fewer variables but with much higher number of constraints.
This is bad news for textbook simplex method so many specialized algorithms were developed,
one of them was simple-minded method of generalized gradient which started the development of non-smooth optimization.

This method relies on subgradient of concave function \(\Phi(\bu, \bv)\)
 which we can transform into convex just by changing signs and replacing \(\inf\) with \(\sup\)
\begin{eqnarray*}
& \Phi(\bu, \bv) = m\sum_{j=1}^n u_j a_j + n\sum_{i=1}^m v_i b_i + \\
& \sum_{i=1}^m \sum_{j=1}^n \sup_{\x \geq 0} x_{ij} \{ c_{ij} + u_j + v_i \} = \\
& = m\sum_{j=1}^n u_j a_j + n\sum_{i=1}^m v_i b_i + \Ind_D(\bu, \bv),
\label{dufun}
\end{eqnarray*}
and ask for its {\em minimization}.

According to convex analysis \cite{TR1970}
the subdifferential \(\cdp \Phi(\bu, \bv) \) exists for any \(\bv, \bu \in \Int \Dom(\Ind_D)\),
and in this case just equals to the (constant) vector
\(g_L = (\bg_\bu, \bg_\bv) = (\ba, \bb)\) of a linear objective in the interior of \(D\).
The situation becomes more complicated when \(\bu, bv\) happens to be at the boundary of \(D\),
the subdifferential set ceases to be a singleton and becomes even unbounded,
roughly speaking certain linear manifolds are added to \(g_L\) but we will not go into details here.
The difficulty is that if we mimic gradient method of the kind
\bleq{sgrad}{
\bu^{k+1} = \bu^{k} - \lambda g_L^u = \bu^{k} - \lambda \ba;
\bv^{k+1} = \bv^{k} - \lambda g_L^v = \bv^{k} - \lambda \bb;
k = 0,1,\dots
}
with a certain step-size \(\lambda > 0\),
we inevitably violate the dual feasibility constraints as \(\ba, \bb > 0.\)
Than the dual function
(\ref{dufun})
becomes undefined and correspondently the subdifferential set becomes undefined as well.

There are at least two simple ways to overcome this difficulty.
One is to incorporate in the gradient method certain operations which restore feasibility
and the appropriate candidate for it is the orthogonal projection operation
where one can make use of the special structure of constraints and sparsity.
However it will still require computing projection operator
of the kind \(B^T(B BT)^{-1}B\) for basis matrices \(B\) with rather uncertain number of iteration and
of matrices of the size around \((n+m)\times(n+m)\).
Neither computers speed nor memory sizes at that time where not up to demands
to solve problems of \(n+m \approx 10^4\) which was required by GOSPLAN!

The second ingenious way was to take into account that if
\(\sum_{j=1}^n a_j = \sum_{i=1}^m b_i = V\),
which is required anyway for solvability of transportation problem in a closed form.
The flow variables may be uniformally bounded by \(V\)
and the dual function can be redefined as
\begin{eqnarray*}
&\Phi_V(\bu, \bv) = m\sum_{j=1}^n u_j a_j + n\sum_{i=1}^m v_i b_i - \\
& \sum_{i=1}^m \sum_{j=1}^n \max_{0 \leq \x \leq V} x_{ij} \{ c_{ij} + u_j + v_i \} = \\
& = m\sum_{j=1}^n u_j a_j + n\sum_{i=1}^m v_i b_i + P_V(\bu,\bv)
\end{eqnarray*}
where the penalty function \(P_V(\bu,\bv)\)
is easily computed by component-wise maximization:
\begin{eqnarray*}
& P_V(\bu, \bv) = \sum_{i=1}^m \sum_{j=1}^n \max_{x_{ij} \in [ 0, V ]} x_{ij} \{ c_{ij} + u_j + v_i \} =
\\
& \sum_{i=1}^m \sum_{j=1}^n V \{ c_{ij} + u_j + v_i \}_+
\end{eqnarray*}
where \(\{\cdot\}_+ = \max\{0, \cdot\}\).
Than the dual objective function becomes finite, the optimization problem --- unconstrained and we can use
simple subgradient method with very low requirements for memory and computations.

Actually even tighter bounds \( \x_{ij} \leq \min(a_i, b_j) \) can be imposed on the flow variables
which may be advantageous for computational reasons.

In both cases there is a fundamental problem of recovering
optimal primal \(n\times m\) primal solution from \(n + m\) dual.
This problem was studied by many authors and recent advances in this area
can be studied from the excellent paper by A. Nedic and A. Ozdoglar
\cite{nedoz09}.
Theoretically speaking, nonzero positive values of \(c_{ij} + u_j^\star + v_i^\star \),
where \(\bu^\star, \bv^\star \) are the {\em exact} optimal solutions of the dual problem
(\ref{dual-trans-prob})
signal that the corresponding optimal primal flow \(x_{ij}^\star\) is equal to zero.
Hopefully after excluding these variables we obtain nondegenerate basis
and can compute the remaining variables
by simple and efficient linear algebra,
especially taking into account the uni-modularity of basis.

However the theoretical gap between zeros and non-zeros
is exponentially small even for modest length integer data
therefore we need an accuracy unattainable by modern 64-128 bits hardware.
Also the real life computations are always accompanied by numerical noise
and we face the hard choice in fact guessing which dual constraints are active and which are not.

To connect the transportation problem with non-smooth optimization notice that
the penalty function \(P_V(\bu, \bv)\) is finite with the subdifferential \(\cdp P_V(\bu, \bv)\)
which can be represented as a set of \(n\times m\)
matrices
\[
{\vec g}_{ij} = \left\{
\begin{array}{ll}
V & \mbox{ if } c_{ij} + u_j + v_i > 0 \\
0 & \mbox{ if } c_{ij} + u_j + v_i < 0 \\
\co(0,V) & \mbox{ if } c_{ij} + u_j + v_i = 0
\end{array}
\right.
\]
so the subdifferential set is a convex hull of up to \(2^{(n+m)}\)
extreme points --- enormous number even for a modest size transportation problem.
Nevertheless it is easy to get at least single member of subdifferential and consider
the simplest version of subgradient method:
\[
\x^{k+1} = \x^k - \lambda \bar \bg^k, k = 0, 1, \dots
\]
where \(\x^0\) is a given starting point, \(\lambda > 0 \) --- fixed step-size and
\(\bar \bg^k = \bg^k/\|\bg^k\|\) is a normalized subgradient \(\bg^k \in \partial f(\x^k)\).
Of course we assume that \(\bg^k \neq 0\) otherwise \(\x^k\) is already a solution.

Of course, there is no hope of classical convergence result such that \(\x^k \to \xs \in \XS\),
but the remarkable theorem of Shor \cite{Shor79} establishes that
this simplest algorithm determines at least the approximate solution.
As a major step in the development of different algorithmic ideas we can start with the subgradient algorithm due
to Shor (see \cite{Shor2012} for the overview and references to earliest works).
Of course, there is no hope of classical convergence result such that \(\x^k \to \xs \in \XS\),
but the remarkable theorem of Shor \cite{Shor79} establishes that
this very simple algorithm provides an approximate solution of (\ref{copt}) at least theoretically.
\begin{theorem}
Let \(f\) is a finite convex function with a subdifferential \(\partial f\)
and the sequence \(\{\xk\}\) is obtained by the recursive rule
\bleq{simgrad}{ x^{k+1} = \xk - \lambda g^k_\nu, k = 0,1, \dots }
with \(\lambda > 0\) and \(g^k_\nu = g^k/\|g^k\|, g^k \in \partial f(\xk),\ g^k \neq 0\) is a normalized subgradient
at the point \(x^k\).
Then for any \(\epsilon > 0\) there is an infinite set \(Z_\epsilon \subset Z\)
such that for any \(k \in Z_\epsilon\)
\[
f(\tilde \xk) = f(\xk) \mbox{ and } \dist(\tilde \xk, X_\star) \leq \lambda(1+\epsilon)/2.
\]
\end{theorem}
The statement of the theorem is illustrated on Fig. \ref{constep}
together with the idea of the proof.
\begin{figure}
\begin{center}
\includegraphics[scale=1.]{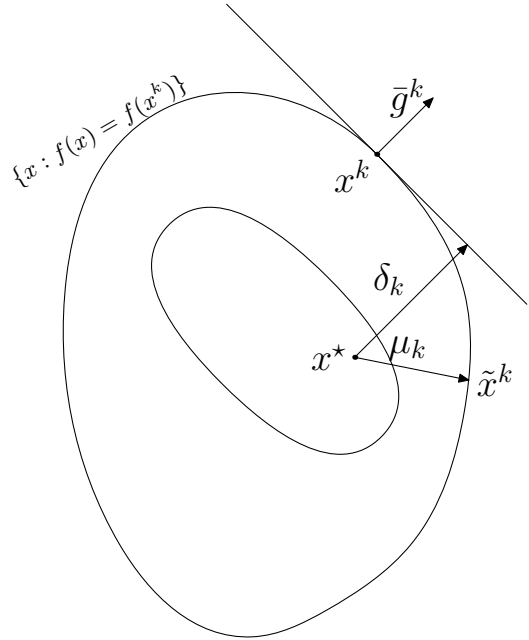}
\end{center}
\caption{The statement and the idea of the proof of Shor theorem}
\label{constep}
\end{figure}
The detailed proof of the theorem goes like following:
Let \(\xs \in \XS\) and estimate
\[
\|\x^{k+1} - \xs \|^2 = \|\xk - \xs - \lambda \bg^k_\nu \|^2 =
\|\xk - \xs \|^2 + \lambda^2 - 2\lambda \sdot{\bar \bg^k}{(\xk - \xs)}.
\]
The last term in fact equals
\[
\begin{array}[t]{c}\min \\ \z \in H_k \end{array}
 \| \xs - \z \|^2 = \| \xs - \z^k\|^2 = \delta_k,
\]
where \( H_k = \{ \z: \sdot{\z}{g^k_\nu} = \sdot{\xk}{g^k_\nu} \)
is a hyperplane, orthogonal to \(g^k_\nu\) and passing through the point \(\xk\),
so
\bleq{del}{ \|\x^{k+1} - \xs \|^2 = \|\xk - \xs \|^2 + \lambda^2 - 2\lambda \delta_k, \ k = 0,1,2, \dots }
If \( \lambda^2 - 2\lambda \delta_k \leq -\lambda^2 \epsilon \) for any \(k\in Z\) then
\bleq{del2}{ \|\x^{k+1} - \xs \|^2 \leq \|\xk - \xs \|^2 - \lambda^2 \epsilon, \ k = 0,1,2, \dots }
therefore \bleq{delk}{ 0 \leq \| \x^{k+1} - \xs \|^2 \leq \| \x^0 - \xs \|^2 \leq - k \lambda^2 \epsilon \to -\infty }
when \(k \to \infty\).
This contradiction proves that there is \(k_0\) such that
\( \lambda^2 - 2\lambda \delta_{k_0} > -\lambda^2 \epsilon \)
or \(\delta_{k_0} < \lambda ( 1 + \epsilon )/2 \).

To complete the proof notice that by convexity \(f(z^{k_0}) \geq f(\x_{k_0})\) and therefore
\bleq{funeq}{
\min_{z: f(z) = f(\x^{k_0})} \| \xs - z \|^2
=
\| \xs - \bar z^{k_0} \|^2
=
\min_{z: f(z) \geq f(x^{k_0})} \| \xs - z \|^2
\leq
\| \xs - z^{k_0} \|^2 = \delta_{k_0}.
}
By setting \(\tilde \x^0 = z^{k_0}\) we obtain \(\| \xs - \tilde \x^0 \|^2 < \lambda ( 1 + \epsilon )/2 \).

By replacing \(x^0\) in (\ref{delk}) by \(\tilde \x^0\) and repeating the reasoning above we obtain \(\tilde x^1\)
such that \(\| \xs - \tilde \x^1 \|^2 < \lambda ( 1 + \epsilon )/2 \), then in the same manner \(\tilde \x^2, \tilde \x^3 \) and so on
with \(\| \xs - \tilde \x^k \|^2 < \lambda ( 1 + \epsilon )/2, k = 2,3, \dots \)
which complete the proof.
\EndProof
\section{Complexity Results for Convex Optimization}
\label{complexity}
At this section we describe the complexity results for non-smooth convex optimization problems.
Most of the results mentioned below can be found in books
\cite{nemirovsky1983problem,polyak1987introduction,nesterov2018lectures,bubeck2015,ben-tal2015lectures}.
We start with the  \ag{\textbf{`small dimensional problems'}, when
$$N\ge n=\dim \x,$$}
where $N$ is a number of oracle calls (number of subgradient calculations
or/and calculations of separation hyperplane to some simple set \ag{at a given point}).

Let's consider convex optimization problem
\begin{equation}
\label{eq1}
f\left( \x \right)\to \mathop {\min }\limits_{\x\in Q},
\end{equation}
where $Q$ -- is a compact and simple set (it's significant here!).
\ag{Based on at least $N$ subgradient calculations (in general, oracle calls) w}e would like to find such a point $\x^N$ that
\[
f\left( {\x^N} \right)-f_\ast \le \varepsilon,
\]
where $f_\ast =f\left( {\x_\ast } \right)$ is an optimal value of function in
(\ref{eq1}), $\x_\ast $ -- the solution of (\ref{eq1}).
The lower and the upper bounds for
the oracle complexity is (\pd{up to a multiplier, which has logarithmic dependence on some characteristic of the set $Q$})
\[
N\sim n\ln \left( {{\Delta f} \mathord{\left/ {\vphantom {{\Delta f}
\varepsilon }} \right. \kern-\nulldelimiterspace} \varepsilon } \right),
\]
where $\Delta f=\mathop {\sup }\limits_{\ag{\x,\y}\in Q} \left\{ {f\left( \y
\right)-f\left( \x \right)} \right\}$.
The center of gravity method \cite{Levin1965,Newman1965} converges according to this estimate.
The center of gravity method in $n=1$ is a simple binary search method \cite{brent1973algorithms}.
But in $n>1$ this method is hard to implement.
The complexity of iteration is
too high, because we required center of gravity oracle \cite{bubeck2015}.
\ag{Well known} ellipsoid method \cite{Shor85,nemirovsky1983problem} requires\footnote{\ag{Here and below for all (large) $n$: $\tilde{ O}(g(n)) \le C\cdot(\ln n)^r g(n)$ with some constants $C>0$ and $r\ge 0$. Typically, $r = 1$. If $r=0$, then $\tilde{O}(\cdot) = O(\cdot)$.}}
$N=\tilde{O}\left( {n^2\ln \left( {{\Delta f} \mathord{\left/
{\vphantom {{\Delta f} \varepsilon }} \right. \kern-\nulldelimiterspace}
\varepsilon } \right)} \right)$ oracle calls and $O\left( {n^2}
\right)$ iteration complexity.
In \cite{Vaidya1989,bubeck2015} a special version
of cutting plane method was proposed.
This method (Vayda's method) requires
$N=\tilde{O}\left( {n\ln \left( {{\Delta f} \mathord{\left/
{\vphantom {{\Delta f} \varepsilon }} \right. \kern-\nulldelimiterspace}
\varepsilon } \right)} \right)$ oracle calls and has iteration complexity
$\tilde {{\rm O}}\left( {n^{2.37}} \right)$.
In the work
\cite{lee2015faster} there proposed a method with $N=\tilde {
O}\left( {n\ln \left( {{\Delta f} \mathord{\left/ {\vphantom {{\Delta f}
\varepsilon }} \right. \kern-\nulldelimiterspace} \varepsilon } \right)}
\right)$ oracle calls and iteration complexity $\tilde {O}\left( {n^2}
\right)$. Unfortunately, for the moment it's not obvious that this method is
very practical one due to the large log-factors in $\tilde {O}\left(
\right)$.

Based on ellipsoid method in the late 70-th Leonid Khachyan showed
\cite{khachiyan1979polynomial} that LP is in P in byte complexity. \pd{Let us shortly explain the idea}.
\pd{The main question is whether} $A\x\le \bb$ is solvable \pd{or not, where $n=\dim \x$, $m=\dim \bb$ and all
elements of $A$ and $\bb$ are integers.}
\pd{We would like also} to find one of the exact
solutions $\x_\ast $.
This problem up to a logarithmic factor in complexity
is equivalent to the problem to find the exact solution of LP problem
$\left\langle {\bc,\x} \right\rangle \to \mathop {\min }\limits_{A\x\le \bb} $
with integer $A$, $\bb$ and $\bc$.
\pd{We consider only inequality constraints as it is known that to} find the exact solution of $A\x=\bb$ one can
use polynomial Gauss \ag{elimination} algorithm with $O\left( {n^3} \right)$ \ag{arithmetic operations (a.o.)} complexity. 

\ag{Let us} introduce
\[
\Lambda =\sum\limits_{i,j=1,1}^{m,n} {\log _2 \left| {a_{ij} } \right|}
+\sum\limits_{i=1}^m {\log _2 \left| {\bb_i } \right|} +\log _2 \left( {mn}
\right)+1.
\]
If $A\x\le \bb_{ }$ is compatible, then there exists such $\x_\ast $ that
$\left\| {\x_\ast } \right\|_\infty \le 2^\Lambda $, $A\x_\ast \le \bb$
otherwise
\[
\mathop {\min }\limits_\x \left\| {(A\x-\bb)_+} \right\|_\infty \ge 2^{-\left(
{\Lambda -1} \right)}.
\]
\pd{Thus, the question of compatibility of $A\x\le \bb$  is equivalent to the problem of finding minimum of the following} 
non-smooth convex optimization problem
\[
\left\| {(A\x-\bb)_{+}} \right\|_\infty \to \mathop {\min }\limits_{\left\| {\x_\ast }
\right\|_\infty \le 2^\Lambda }.
\]
\pd{The approach of \cite{khachiyan1979polynomial}}  is to apply ellipsoid method for this problem with $\varepsilon
=2^{-\Lambda }$.
\pd{From the complexity of this method, it follows that } in $O\left( {n\Lambda } \right)$-bit arithmetic with $\tilde
{O}\left( {mn+n^2} \right)$ cost of PC memory one can find $\x_\ast $
(if \pd{it exists}) \pd{in} $\tilde {O}\left( {n^3\left( {n^2+m}
\right)\Lambda } \right)$ a.o.

\ag{N}ote, that in the ideal arithmetic with real
numbers it is still an open question \cite{blum2012complexity} \pd{whether} it \pd{is}
possible to find the exact solution of LP \pd{an} problem (with \pd{the data given by } real numbers) in
polynomial time in the ideal arithmetic ($\pi \cdot e$ -- costs $O\left( 1
\right))$. 

\ag{Now let us consider \textbf{`large dimensional problems'}
$$N\le n=\dim \x.$$} Table \ref{tabl} describes (for more details see
\cite{ben-tal2015lectures,bubeck2015,nesterov2018lectures})
optimal estimates for the number of oracle
calls for convex optimization problem (\ref{eq1}) in the case
when \( N\le n. \)
Now $Q$ is not necessarily compact set.

\begin{table}[htbp]
\caption{Optimal estimates for the number of oracle calls}
\label{tabl}
\begin{center}
\begin{tabular}{|p{100pt}|p{100pt}|p{114pt}|}
\hline
$N\le n$&
$\left| {f\left( \y \right)-f\left( \x \right)} \right|\le M\left\| {\y-\x} \right\|$&
$\left\| {\nabla f\left( \y \right)-\nabla f\left( \x \right)} \right\|_\ast \le L\left\| {\y-\x} \right\|$ \\
\hline
$f\left( \x \right)$ convex&
$O\left( {\frac{M^2R^2}{\varepsilon ^2}} \right)$&
$O\left( {\sqrt {\frac{LR^2}{\varepsilon }} } \right)$ \\
\hline
$f\left( \x \right)
\quad
\mu -$strongly convex in \ag{$\left\|\cdot \right\|$}-norm&
$\tilde {O}\left( {\frac{M^2}{\mu \varepsilon }} \right)$&
$\tilde {O}\left( {\sqrt {\frac{L}{\mu }} \left\lceil {\ln \left( {\frac{\mu R^2}{\varepsilon }} \right)} \right\rceil } \right)
\quad
\left( {\forall \;N} \right)$ \\
\hline
\end{tabular}
\end{center}
\end{table}
Here $R$ is a ``distance'' (up to a $\ln n$-factor) between starting point
and the nearest solution
\[
R=\tilde {O}\left( {\left\| {\x^0-\x_\ast } \right\|} \right).
\]
Let's describe optimal method in the most simple case: \ag{$Q={\R}^n$,
$\left\|\cdot \right\|=\left\|\cdot \right\|_2$} \cite{polyak1987introduction,nesterov2009primal-dual}.
Define
\[
B_2^n \left( {\x_\ast,R} \right)=\left\{ {\x\in {\ag{\R}}^n:\;\;\left\|
{\x-\x_\ast } \right\|_2 \le R} \right\}.
\]
The main iterative process is (for simplicity we'll denote arbitrary element
of $\partial f\left( \x \right)$ as $\nabla f\left( \x \right))$
\begin{equation}
\label{eq2}
\x^{k+1}\mbox{=}\;\x^k-h\nabla f\left( {\x^k} \right).
\end{equation}
Assume that under $\x\in B_2^n \left( {\x_\ast,\sqrt 2 R} \right)$
\begin{equation}
\label{eq3}
\left\| {\nabla f\left( \x \right)} \right\|_2 \le M,
\end{equation}
where $R=\left\| {\x^0-\x_\ast } \right\|_2 $.

Hence, from (\ref{eq2}), (\ref{eq3}) we have
\[
\left\| {\x-\x^{k+1}} \right\|_2^2 \mbox{=}\left\| {\x-\x^k+h\nabla f\left(
{\x^k} \right)} \right\|_2^2 =
\]
\[
=\;\left\| {\x-\x^k} \right\|_2^2 +2h\left\langle {\nabla f\left( {\x^k}
\right),\x-\x^k} \right\rangle +h^2\left\| {\nabla f\left( {\x^k} \right)}
\right\|_2^2 \le
\]
\ag{
\begin{equation}
\label{KI}
\le \;\left\| {\x-\x^k} \right\|_2^2 +2h\left\langle {\nabla f\left( {\x^k}
\right),\x-\x^k} \right\rangle +h^2M^2.
\end{equation}
}
Here we choose $\x=\x_\ast $ (if $\ag{\x}_\ast $ isn't unique, we choose the nearest
$\x_\ast $ to $\x^0)$
\[
f\left( {\frac{1}{N}\sum\limits_{k=0}^{N-1} {\x^k} } \right)-f_\ast \le
\frac{1}{N}\sum\limits_{k=0}^{N-1} {f\left( {\x^k} \right)} -f\left( {\x_\ast
} \right)\le \frac{1}{N}\sum\limits_{k=0}^{N-1} {\left\langle {\nabla
f\left( {\x^k} \right),\x^k-\x_\ast } \right\rangle } \le
\]
\[
\le \frac{1}{2hN}\sum\limits_{k=0}^{N-1} {\left\{ {\left\| {\x_\ast -\x^k}
\right\|_2^2 -\left\| {\x_\ast -\x^{k+1}} \right\|_2^2 } \right\}}
+\frac{hM^2}{2}=
\]
\[
=\frac{1}{2hN}\left( {\left\| {\x_\ast -\x^0} \right\|_2^2 -\left\| {\x_\ast
-\x^N} \right\|_2^2 } \right)+\frac{hM^2}{2}.
\]
If
\begin{equation}
\label{eq4}
h=\frac{R}{M\sqrt N },
\quad
\bar {\x}^N=\frac{1}{N}\sum\limits_{k=0}^{N-1} {\ag{\x}^k},
\end{equation}
then
\begin{equation}
\label{eq5}
f\left( {\bar {\x}^N} \right)-f_\ast \le \frac{MR}{\sqrt N}.
\end{equation}
Note that the precise lower bound for fixed steps first-order methods for the class of convex optimization problems with (\ref{eq3}) \cite{Drori-Teboulle2016}
\[
f\left( {\x^N} \right)-f_\ast \ge \frac{MR}{\sqrt {N+1} }.
\]
Inequality (\ref{eq5}) means that (see also Table 1)
\[
N=\frac{M^2R^2}{\varepsilon ^2},
\quad
h=\frac{\varepsilon }{M^2}.
\]
So, one can mentioned that if we will use in (\ref{eq2})
\begin{equation}
\label{eq6}
\x^{k+1}\mbox{=}\;\x^k-h_k \nabla f\left( {\x^k} \right),
\quad
h_k =\;\frac{\varepsilon }{\left\| {\nabla f\left( {\x^k} \right)}
\right\|_2^2 }
\end{equation}
the result (\ref{eq5}) \ag{holds with} \cite{nesterov2009primal-dual}
\[
\bar {\x}^N=\frac{1}{\sum\limits_{k=0}^{N-1} {h_k } }\sum\limits_{k=0}^{N-1}
{h_k \x^k}.
\]
If we put in (\ref{eq6}),
\[
h_k =\frac{R}{\left\| {\nabla f\left( {\x^k} \right)} \right\|_2 \sqrt N },
\]
like in (\ref{eq4}), \pd{the result similar} to (\ref{eq5}) also \ag{holds}
\[
\mathop{ \min\limits_{k=0,...,N-1} f\left( {\x^k} \right)-f_\ast \le\frac{MR}{\sqrt N }}
\]
not only for the convex functions, but also for quasi-convex functions
\cite{Polyak1969,nesterov1989}:

\[
f(\alpha \x + (1-\alpha)\y) \leq \max \left\{f(\x), f(\y) \right\} \text{ for all } \x, \y \in Q, \alpha \in
[0,1].
\]

Note that
\[
0\le \frac{1}{2hk}\left( {\left\| {\x_\ast -\x^0} \right\|_2^2 -\left\|
{\x_\ast -\x^k} \right\|_2^2 } \right)+\frac{hM^2}{2}\pd{.}
\]
Hence, for all $k=0,...,N$,
\[
\left\| {\x_\ast -\x^k} \right\|_2^2 \le \left\| {\x_\ast -\x^0} \right\|_2^2
+h^2M^2k\le 2\left\| {\x_\ast -\x^0} \right\|_2^2,
\]
therefore
\begin{equation}
\label{eq7}
\left\| {\x^k-\x_\ast } \right\|_2 \le \sqrt 2 \left\| {\x^0-\x_\ast }
\right\|_2,
\quad
k=0,...,N.
\end{equation}
Inequality (\ref{eq7}) justifies that we need assumption (\ref{eq3}) \ag{holds} only with
$\x\in B_2^n \left( {\x_\ast,\sqrt 2 R} \right)$.

For the general (constrained) case (\ref{eq1}) we introduce \pd{a} norm \ag{$\left\|\cdot \right\|$} \pd{and some}
prox-function $d\left( \x \right)\ge 0$, which is \pd{continuous and}
1-strongly convex \pd{with respect to} \ag{$\left\|\cdot\right\|$}\pd{, i.e. $d(y)-d(x) - \la d(x),y-x \ra \geq \frac{1}{2}\|x-y\|^2$, for all $x,y \in Q$.}
\pd{We also introduce} Bregman's divergence
\cite{ben-tal2015lectures}
\[
\ag{V[\x](\y)=d\left( \y \right)-d\left( \x \right)-\left\langle
{\nabla d\left( \x \right),\y-\x} \right\rangle.}
\]
We \pd{set} \ag{$R^2=V[\x^0](\x_{\ast})$}, where $\x_\ast $ -- is solution
of (\ref{eq1}) (if $\x_\ast $ isn't unique then we assume that $\x_\ast $ is minimized
\ag{$V[\x^0](\x_{\ast})$}. \ag{The natural generalization of iteration process \eqref{eq2} is Mirror Descent algorithm \cite{nemirovskii1979efficient,ben-tal2015lectures}} which iterates as
\[
\x^{k+1}\mbox{=}\;\mbox{Mirr}_{\x^k} \left( {h\nabla f\left( {\x^k} \right)}
\right),
\quad
\mbox{Mirr}_{\x^k} \left( \mbox{v} \right)=\arg \mathop {\min }\limits_{\x\in
Q} \left\{ {\left\langle {\mbox{\bv},\x-\x^k} \right\rangle +\pd{V[\x^k]\left( {\x}
\right)}} \right\}.
\]

\ag{For this iteration process} instead of \ag{\eqref{KI}} we have
\[
2\pd{V[\x^{k+1}]\left( \x\right)}\le \;2\pd{V [\x^k]\left(\x \right)}+2h\left\langle
{\nabla f\left( {\x^k} \right),\x-\x^k} \right\rangle +h^2M^2,
\]
\ag{where $\left\|
{\nabla f\left( \x \right)} \right\|_\ast \le M$ for all $x: \pd{V[\x](\x_*) \le 2V[\x^0](\x_*)} \ag{= 2R^2}$} \pd{, see also Section \ref{large-dim}.}

\ag{An}alogues of formulas (\ref{eq4}), (\ref{eq5}), \ag{\eqref{eq7}} are also valid
\[
f\left( {\bar {\x}^N} \right)-f_\ast \le \frac{\sqrt 2 MR}{\sqrt N }, 
\]
where
$$\ag{\quad \bar {\x}^N=\frac{1}{N}\sum\limits_{k=0}^{N-1} {\x^k}, \quad h=\frac{\varepsilon }{M^2}}$$
\ag{and}
\[
\left\| {\x^k-\x_\ast } \right\|\le 2R\ag{,  \quad
k=0,...,N.}
\]
\ag{In \cite{ben-tal2015lectures} authors discus how to choose $d(\x)$ for different simple convex sets $Q$. One of these examples (unit simplex) will considered below. Note, that in all these examples one can guarantees that \cite{ben-tal2015lectures}:
\[
R\le C\sqrt{\ln n}\cdot \left\|
{\x_\ast -\x^0} \right\|.
\]
}
\ag{
Note, that if $Q={\R}^n$,
$\left\|\cdot \right\|=\left\|\cdot \right\|_2$ then $d(\x) = \frac{1}{2}\|\x\|_2^2$, $V[\x](\y)=\frac{1}{2}\|\y-\x\|_2^2$,
$$\x^{k+1} = \mbox{Mirr}_{\x^k} \left(h\nabla f\left( {\x^k} \right) \right)=\arg \mathop {\min }\limits_{\x\in
\R^n} \left\{h {\left\langle {\nabla f\left( {\x^k} \right),\x-\x^k} \right\rangle +\frac{1}{2}\|\x-\x^k\|_2^2} \right\} =$$
$$= \x^k-h \nabla f\left( {\x^k} \right),$$
that corresponds to the standard gradient-type iteration process~\eqref{eq2}.}

\textbf{Example (unit simplex).} \textit{We have
\[
Q=S_n \left( 1 \right)=\left\{ {\x\in {\rm R}_+^n:\;\;\sum\limits_{i=1}^n
{\x_i } =1} \right\},
\quad
\left\| {\nabla f\left( \x \right)} \right\|_\infty \le M_\infty,
\quad
\x\in Q,
\]
\[
\ag{\left\|\cdot \right\|=\left\|\cdot \right\|_1},
\quad
d\left( \x \right)=\ln n+\sum\limits_{i=1}^n {\x_i \ln \x_i },
\quad
h=M_\infty ^{-1} \sqrt {{2\ln n} \mathord{\left/ {\vphantom {{2\ln n} N}}
\right. \kern-\nulldelimiterspace} N},
\quad
\x_i^0 =1 \mathord{\left/ {\vphantom {1 n}} \right.
\kern-\nulldelimiterspace} n,
\quad
i=1,...,n.
\]
For $k=0,...,N-1$, $i=1,...,n$
\[
\x_i^{k+1} =\frac{\exp \left( {-h\sum\limits_{r=1}^k {\nabla _i f\left( {\x^r}
\right)} } \right)}{\sum\limits_{l=1}^n {\exp \left( {-h\sum\limits_{r=1}^k
{\nabla _l f\left( {\x^r} \right)} } \right)} }=\frac{\x_i^k \exp \left(
{-h\nabla _i f\left( {\x^k} \right)} \right)}{\sum\limits_{l=1}^n {\x_l^k \exp
\left( {-h\nabla _l f\left( {\x^k} \right)} \right)} }.
\]
The main result here is
\[
f\left( {\bar {\x}^N} \right)-f_\ast \le M_\infty \sqrt {\frac{2\ln n}{N}},
\quad
\bar {\x}^N=\frac{1}{N}\sum\limits_{k=0}^{N-1} {\x^k}.
\]
Note, that if we use \ag{$\left\|\cdot\right\|_2$}-norm and $d\left( \x
\right)=\frac{1}{2}\left\| {\x-\x^0} \right\|_2^2$ here, we will have \ag{higher
iteration complexity} (2-norm projections on unit simplex) and
\[
f\left( {\bar {\x}^N} \right)-f_\ast \le \frac{M_2 }{\sqrt N },
\quad
\left\| {\nabla f\left( \x \right)} \right\|_2 \le M_2,
\quad
\x\in Q.
\]
Since typically $M_2 ={\rm O}\left({\sqrt n M_\infty } \right)$, it is worth
to use \ag{$\left\|\cdot \right\|_1$}-norm.}

Assume now that $f(x)$ in (\ref{eq1}) is additionally $\mu $-strongly
convex in \ag{$\left\|\cdot \right\|_2$} norm:

\ag{
\[
f(\y) \ge f(\x) + \la \nabla f(\x), y-x \ra +  \frac{\mu }{2}\left\|\y-\x\right\|_2^2   \text{ for all } \x,\y\in Q.
\]
}

Let
\[
\x^{k+1}\mbox{=}\;\mbox{Mirr}_{\x^k} \left( {h_k \nabla f\left( {\x^k} \right)}
\right)=\arg \mathop {\min }\limits_{\x\in Q} \left\{ {h_k \left\langle
{\nabla f\left( {\x^k} \right),\x-\x^k} \right\rangle +\frac{1}{2}\left\|
{\x-\x^k} \right\|_2^2 } \right\},
\]
where
\[
h_k =\frac{2}{\mu \cdot \left( {k+1} \right)},
\quad
d\left( \x \right)=\frac{1}{2}\left\| {\x-\x^0} \right\|_2^2,
\quad
\left\| {\nabla f\left( \x \right)} \right\|_2 \le M,
\quad
\x\in Q.
\]
Then \cite{Lacost-Julien2012}
\[
f\left( {\sum\limits_{k=1}^N {\frac{2k}{k\left( {k+1} \right)}\x^k} }
\right)-f_\ast \le \frac{2M^2}{\mu \cdot \left( {k+1} \right)}.
\]
Hence (see also Table \ref{tabl}),
\[
N\simeq \frac{2M^2}{\mu \varepsilon }.
\]
This bound is also un-improvable up to a constant factor \cite{nemirovsky1983problem,nesterov2018lectures}.
\section{Looking into the Black-Box}
\label{black-box}
In this section we consider \pd{how problem special structure can} be used to solve non-smooth optimization problems with the convergence rate $O\left( \frac{1}{k}\right)$, which is faster than the lover bound $O\left( \frac{1}{\sqrt{k}}\right)$ for general class of non-smooth convex problems \cite{nemirovsky1983problem}. Nevertheless, there is no contradiction as additional structure is used and we are looking inside the black-box.
\subsection{Nesterov's smoothing}
In this subsection, following \cite{nesterov2005smooth}, we consider the problem
\begin{equation}
 \label{eq:Sm-ng-Problem}
 \min_{\x\in Q_1 \subset E_1} \{ f(\x) = h(\x) + \max_{ \bu \in Q_2 \subset E_2} \{ \iprod{A\x}{\bu} - \phi(\bu) \} \},
\end{equation}
where $A: E_1 \to E_2^*$ is a linear operator, $\phi(\bu)$ is a continuous convex function on $Q_2$, $Q_1, Q_2$ are convex compacts, $h$ is convex function with $L_h$-Lipschitz-continuous gradient.

Let us consider an example of $f(x) = \|A \x - \bb\|_{\infty}$ with $A \in \R^{m \times n}$. Then,
\[
f(x) = \max_{\bu \in \R^m} \left\{ \iprod{\bu}{A \x - \bb}: \|\bu\|_1 \leq 1 \right\},
\]
$h = 0$, $E_2 = \R^m$, $\phi(\bu) = \iprod{\bu}{\bb}$ and $Q_2$ is the ball in 1-norm.

The main idea of Nesterov is to add regularization inside the definition of $f$ in \eqref{eq:Sm-ng-Problem}. More precisely, a prox-function $d_2(\bu)$ (see definition in Section \ref{complexity}) is introduced for the set $Q_2$ and a smoothed counterpart $f_{\mu}(\x)$ for $f$ is defined as
\[
f_{\mu}(\x) = h(\x) + \max_{ \bu \in Q_2} \{ \iprod{A\x}{\bu} - \phi(\bu) - \mu d_2(\bu) \}
\]
and $\bu_{\mu}(\x)$ is the optimal solution of this maximization problem.
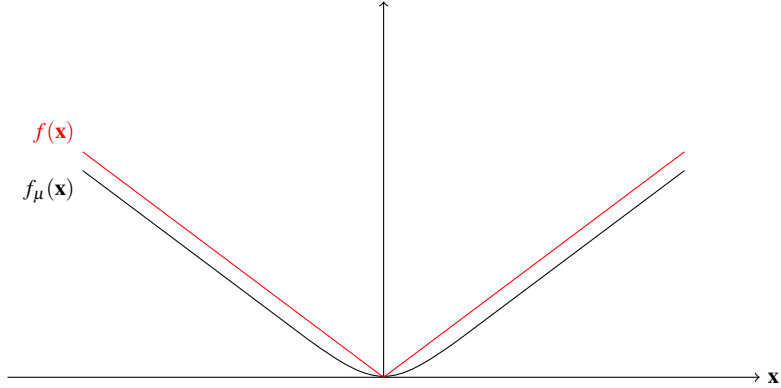
\begin{figure}
\centering
\begin{tikzpicture}[scale=1]
 \coordinate (y) at (0,5);
 \coordinate (x) at (5,0);
 \draw[<->] (y) node[above] {} -- (0,0) -- (x) node[right]
 {$\x$};
	\draw[-] (-5,0) -- (0,0);

 \path
	coordinate (start) at (-1,.5)
 coordinate (c1) at +(-.1,-0.15)
 coordinate (c2) at +(.1,-0.15)
 coordinate (slut) at (1,.5)
 coordinate (top) at (4.2,2);

 \draw (start).. controls (c1) and (c2).. (slut);
		
		\draw[-] (-4,2.75) node[below left] {$f_{\mu}(\x)$} -- (-1,0.5);
		\draw[-] (4,2.75) -- (1,0.5);
		
		\draw[red,-] (-4,3) node[above left] {{\color{red}$f(\x)$}} -- (0,0);
		\draw[red,-] (4,3) -- (0,0);
		
 \end{tikzpicture}
 \caption{Function $f_{\mu}(\x)$ is a smooth approximation to non-smooth function $f(\x)$.}
\end{figure}

\begin{theorem}[\cite{nesterov2005smooth}]
\label{Th:Sm-ng-Diff-f-mu}
The function $f_{\mu}(\x)$ is well defined, convex and continuously differentiable at any $\x \in E_1$ with $\nabla f_{\mu}(\x) = \nabla h(\x) + A^* \bu_{\mu}(\x)$. Moreover, $\nabla f_{\mu}(\x)$ is Lipschitz continuous with constant $L_{\mu} = L_h+\frac{\|A\|_{1,2}^2}{\mu}$.
\end{theorem}
Here the adjoint operator $A^*$ is defined by equality $\iprod{A\x}{\bu}=\iprod{A^*\bu}{\x}$, $\x\in E_1, \bu \in E_2$ and the norm of the operator $\|A\|_{1,2}$ is defined by $\|A\|_{1,2} = \max_{\x,\bu}\{\iprod{A\x}{\bu}:\|\x\|_{E_1} = 1, \|\bu\|_{E_2} = 1 \}$.

Since $Q_2$ is bounded, $f_{\mu}(\x)$ is a uniform approximation for the function $f$, namely, for all $\x \in Q_1$,
\begin{equation}
\label{eq:Sm-ng-Unif-Appr}
 f_{\mu}(\x) \leq f(\x) \leq f_{\mu}(\x) + \mu D_2,
\end{equation}
where $D_2 = \max \{d_2(\bu): \bu \in Q_2 \}$.

Then, the idea is to choose $\mu$ sufficiently small and apply accelerated gradient method to minimize $f_{\mu}(\x)$ on $Q_1$. We use accelerated gradient method from \cite{dvurechensky2017adaptive,dvurechensky2018computational} which is different from the original method of \cite{nesterov2005smooth}.

\begin{algorithm}[h!]
\caption{Accelerated Gradient Method}
\label{Alg:Sm-ng-AGM}
{
\begin{algorithmic}[1]
		\REQUIRE Objective $f(\x)$, feasible set $Q$, Lipschitz constant $L$ of the $\nabla f(\x)$, starting point $\x^0 \in Q$,, prox-setup: $d(\x)$ -- $1$-strongly convex w.r.t. $\|\cdot\|_{E_1}$, $V[\z] (\x):= d(\x) - d(\z) - \la \nabla d(\z), \x - \z \ra$.
		\STATE Set $k=0$, $C_0=\alpha_0=0$, $\y^0=\z^0=\x^0$.
			\FOR{$k=0,1,...$}
				\STATE Find $\alpha^{k+1}$ as the largest root of the equation
				\begin{equation}
				C_{k+1}:=C_k+\alpha_{k+1} = L\alpha_{k+1}^2.
				\label{eq:Sm-ng-alpQuadEq}
				\end{equation}
				\STATE
				\begin{equation}
				\x^{k+1} = \frac{\alpha_{k+1}\z^{k} + C_k \y^{k}}{C_{k+1}}.
				\label{eq:Sm-ng-lambdakp1Def}
				\end{equation}
				\STATE
				\begin{equation}
				\z^{k+1}=\arg \min_{\x \in Q} \{V[\z^{k}](\x) + \alpha_{k+1}(f(\x^{k+1}) + \langle \nabla f(\x^{k+1}), \x - \x^{k+1} \rangle) \}.
				\label{eq:Sm-ng-zetakp1Def}
				\end{equation}
				\STATE
				\begin{equation}
				\y^{k+1} = \frac{\alpha_{k+1}\z^{k+1} + C_k \y^{k}}{C_{k+1}}.
				\label{eq:Sm-ng-etakp1Def}
				\end{equation}
			\STATE Set $k=k+1$.
		\ENDFOR
		\ENSURE The point $\y^{k+1}$.	
\end{algorithmic}
}
\end{algorithm}

\begin{theorem}[\cite{dvurechensky2017adaptive,dvurechensky2018computational}]
\label{Th:Sm-ng-AGM-Conv}
Let the sequences $\{\x^{k}, \y^{k}, \z^{k}, \alpha_k, C_k \}$, $k\geq 0$ be generated by Algorithm \ref{Alg:Sm-ng-AGM}. Then, for all $k \geq 0$, it holds that
\begin{equation}
 	\label{eq:Sm-ng-AGM-Conv}
 	f(\y^{k}) - f^* \leq \frac{4LV[\z_0](\x^\star)}{(k+1)^2}.
\end{equation}
\end{theorem}

Following the same steps as in the proof of Theorem 3 in \cite{nesterov2005smooth}, we obtain
\begin{theorem}
\label{Th:Sm-ng-Main-Th}
Let Algorithm \ref{Alg:Sm-ng-AGM} be applied to minimize $f_{\mu}(\x)$ on $Q_1$ with $\mu = \frac{2\|A\|_{1,2}}{N+1} \sqrt{\frac{D_1}{D_2}}$, where $D_1 = \max \{d_1(\x): \x \in Q_1 \}$. Then, after $N$ iterations, we have
\begin{equation}
\label{eq:Sm-ng-Main-Conve-Rate}
 0 \leq f(\y^N) - f_\star \leq \frac{4\|A\|_{1,2} \sqrt{D_1D_2}}{N+1} + \frac{4L_hD_1}{(N+1)^2}.
\end{equation}
\end{theorem}
\begin{proof}
Applying Theorem \ref{Th:Sm-ng-AGM-Conv} to $f_{\mu}$, and using \eqref{eq:Sm-ng-Unif-Appr}, we obtain

\begin{align*}
 0 \leq f(\y^N) - f_\star & \leq f_{\mu}(\y^N) + \mu D_2 - f_{\mu}(\x_{\mu}^\star) \leq \mu D_2 + \frac{4L_{\mu}D_1}{(N+1)^2} + \frac{4L_hD_1}{(N+1)^2}
\\
& = \mu D_2 + \frac{4\|A\|_{1,2}^2D_1}{\mu(N+1)^2} + \frac{4L_hD_1}{(N+1)^2}.
\end{align*}

Substituting the value of $\mu$ from the theorem statement, we finish the proof.
\EndProof
\end{proof}

A generalization of the smoothing technique for the case of non-compact sets $Q_1, Q_2$, which is especially interesting when dealing with problems dual to problems with linear constraints, can be found in \cite{tran-dinh2015smooth}. Ubiquitous entropic regularization of optimal transport \cite{cuturi2013sinkhorn} can be seen as a particular case of the application of smoothing technique, especially in the context of Wasserstein barycenters \cite{cuturi2014fast,uribe2018distributed,dvurechensky2018decentralize}.
\subsection{Nemirovski's Mirror Prox}
\label{nemir-prox}
In his paper \cite{nemirovski2004prox}, Nemirovski considers problem \eqref{eq:Sm-ng-Problem} in the following form
\begin{equation}
 \label{eq:MP-Problem}
 \min_{\x\in Q_1 \subset E_1} \{ f(\x) = h(\x) + \max_{ \bu \in Q_2 \subset E_2} \{ \iprod{A\x}{\bu} + \iprod{\bb}{\bu} \} \},
\end{equation}
pointing to the fact that this problem is as general as \eqref{eq:Sm-ng-Problem}. Indeed, the change of variables $\bu \leftarrow (\bu,t)$ and the feasible set $Q_2 \leftarrow \{(\bu,t): \min_{\bu' \in Q_2} \phi(\bu') \leq t \leq \phi(\bu)\} $ allows to make $\phi$ linear. His idea is to consider problem \eqref{eq:MP-Problem} directly as a convex-concave saddle point problem and associated weak variational inequality (VI).
\begin{equation}
 \label{eq:MP-VI}
 \text{Find} \;\;\; \z^\star = (\x^\star, \bu^\star) \in Q_1 \times Q_2 \;\;\; \text{s.t.} \;\;\; \la \Phi(\z), \z^\star - \z \ra \leq 0 \;\; \forall \z \in Q_1 \times Q_2,
\end{equation}
where the operator
\begin{equation}
 \label{eq:MP-Oper}
\Phi(\z) = \left(\begin{array}{c} \nabla h(\x) + A^* \bu \\ -A \x - \bb \end{array} \right)
\end{equation}
is monotone, i.e. $\iprod{\Phi(\z_1)- \Phi(\z_2)}{\z_1 - \z_2} \geq 0$, and Lipschitz-continuous, i.e. $\|\Phi(\z_1)- \Phi(\z_2)\|_* \leq L \|\z_1 - \z_2\|$. With the appropriate choice of norm on $E_1 \times E_2$ and prox-function for $Q_1 \times Q_2$, see Section 5 in \cite{nemirovski2004prox}, the Lipschitz constant for $\Phi$ can be estimated as $L= 2 \|A\|_{1,2}\sqrt{D_1D_2}+ L_hD_1$.

\begin{algorithm}[ht]
\caption{Mirror Prox}
\label{Alg:Sm-ng-UMP}
\begin{algorithmic}[1]
 \REQUIRE General VI on a set $Q \subset E$ with operator $\Phi(\z)$, Lipschitz constant $L$ of $\Phi(\z)$, prox-setup: $d(\z)$, $V[\z] (\bw)$.
 \STATE Set $k=0$, $\z^0 = \arg \min_{\z \in Q} d(\z)$.
 \FOR{$k=0,1,...$}
		\STATE Calculate
				\begin{equation}
				\bw^k={\mathop {\arg \min }\limits_{\z\in Q}}\left\{\la \Phi(\z^k),\z \ra + LV[\z^k](\z) 		 \right\}.
				\label{eq:MP-MPwStep}
				\end{equation}
				\STATE Calculate
				\begin{equation}
				\z^{k+1}={\mathop {\arg \min }\limits_{\z\in Q}} \left\{\la \Phi(\bw^k),\z \ra + LV[\z^k](\z) 		 \right\}.
				\label{eq:MP-MPzStep}
				\end{equation}
			\STATE Set $k=k+1$.
	\ENDFOR
		\ENSURE $\widehat{\bw}^k = \frac{1}{k}\sum_{i=0}^{k-1}\bw^i$.
\end{algorithmic}
\end{algorithm}

\begin{theorem}[\cite{nemirovski2004prox}]
\label{Th:MP-MPGenRate}
Assume that $\Phi(\z)$ is monotone and $L$-Lipschitz-continuous. Then, for any $k \geq 1$ and any $\bu \in Q$,
\begin{equation}
\label{eq:MP-MPGenRate}
\max_{\z \in Q} \la \Phi(\z), \widehat{\bw}^k - \z \ra \leq \frac{L}{k} \max_{\z \in Q}V[\z^0](\z).
\end{equation}
Moreover, if the VI is associated with a convex-concave saddle point problem, i.e.
\begin{itemize}
 \item $E = E_1 \times E_2$,
 \item $Q = Q_1 \times Q_2$ with convex compact sets $Q_1 \subset E_1$, $Q_2 \subset E_2$
 \item $\Phi(\z) = \Phi(\x,\bu) = \left(\begin{array}{c} \nabla_{\x} f(\x, \bu) \\ -\nabla_{\bu} f(\x, \bu) \end{array} \right)$ for a continuously differentiable
function $f(\x, \bu)$ which is convex in $\x \in Q_1$ and concave in $\bu \in Q_2$,
\end{itemize}
then
\begin{equation}
\label{eq:MP-MPSPRate}
[\max_{\bu \in Q_2} f(\widehat{\x}^k, \bu) - \min_{\x \in Q_1}\max_{\bu \in Q_2} f(\x, \bu)] + [\min_{\x \in Q_1}\max_{\bu \in Q_2} f(\x, \bu) - \min_{\x \in Q_1} f(\x, \widehat{\bu}^k) ] \leq \frac{L}{k} \max_{\z \in Q}V[\z^0](\z).
\end{equation}

\end{theorem}

Choosing appropriately the norm in the space $E_1 \times E_2$ and applying Mirror Prox algorithm to solve problem \eqref{eq:MP-Problem} as a saddle point problem, we obtain that the saddle point error in the l.h.s. of \eqref{eq:MP-MPSPRate} decays as $\frac{2 \|A\|_{1,2}\sqrt{D_1D_2}+L_hD_1}{k}$. This is slightly worse than the rate in \eqref{eq:Sm-ng-AGM-Conv} since the accelerated gradient method allows the faster decay for the smooth part $h(\x)$. An accelerated Mirror Prox method with the same rate as in \eqref{eq:Sm-ng-AGM-Conv} can be found in \cite{chen2017accelerated}.

 \section{Non-Smooth Optimization in Large Dimensions}
\label{large-dim}
The optimization of non-smooth functionals with constraints attracts widespread interest
in large-scale optimization and its applications \cite{ben-tal1997robust,nesterov2014primal-dual}.
Subgradient methods for non-smooth optimization have a long history
starting with the method for deterministic unconstrained problems and
Euclidean setting in \cite{shor1967generalized} and
the generalization for constrained problems in \cite{polyak1967general},
where the idea of steps switching between the direction of subgradient of the objective and
the direction of subgradient of the constraint was suggested.
Non-Euclidean extension, usually referred to as Mirror Descent,
originated in \cite{nemirovskii1979efficient,nemirovsky1983problem} and
was later analyzed in \cite{beck2003mirror}.
An extension for constrained problems was proposed in \cite{nemirovsky1983problem},
see also recent version in \cite{beck2010comirror}.
To prove faster convergence rate of Mirror Descent for
strongly convex objective in an unconstrained case,
the restart technique \cite{nemirovskii1985optimal,nemirovsky1983problem,nesterov1983method} was used in \cite{juditsky2012first-order}.
Usually, the step-size and stopping rule for Mirror Descent requires
to know the Lipschitz constant of the objective function and constraint, if any.
Adaptive step-sizes, which do not require this information,
are considered in \cite{nemirovskii1979efficient} for problems without inequality constraints,
and in \cite{beck2010comirror} for constrained problems.

Formally speaking, we consider the following convex constrained minimization problem
\begin{align}
\label{eq:PrSt}
 \min \{ f(\x): \quad \x \in X \subset E, \quad g(\x) \leq 0\},
\end{align}
where $X$ is a convex closed subset of a finite-dimensional real vector space $E$, $f: X \to \R$, $g: E \to \R$ are convex functions.

We assume $g$ to be a non-smooth Lipschitz-continuous function and the problem \eqref{Alg:MDNS} to be regular. The last means that there exists a point $\bar{\x}$ in relative interior of the set $X$, such that $g(\bar{\x}) < 0$.

Note that, despite problem \eqref{eq:PrSt} contains only one inequality constraint, considered algorithms allow to solve more general problems with a number of constraints given as $\{ g_i(\x) \leq 0, i=1,...,m \}$. The reason is that these constraints can be aggregated and represented as an equivalent constraint given by $\{ g(\x) \leq 0\}$, where $g(\x)=\max_{i=1,...,m}g_i(\x)$.

We consider two adaptive Mirror Descent methods \cite{bayandina2018mirror} for the problem \eqref{eq:PrSt}. Both considered methods have complexity $O\left(\frac{1}{\varepsilon^2}\right)$ and optimal.

We consider algorithms, which are based on Mirror Descent method. Thus, we start with the description of proximal setup and basic properties of Mirror Descent step.
Let $E$ be a finite-dimensional real vector space and $E^*$ be its dual. We denote the value of a linear function $g \in E^*$ at $\x\in E$ by $\langle \bg, \x \rangle$. Let $\|\cdot\|_{E}$ be some norm on $E$, $\|\cdot\|_{E,*}$ be its dual, defined by $\|\bg\|_{E,*} = \max\limits_{\x} \big\{ \langle \bg, \x \rangle, \| \x \|_E \leq 1 \big\}$. We use $\nabla f(\x)$ to denote any subgradient of a function $f$ at a point $\x \in {\rm dom} f$.

Given a vector $\x\in X^0$, and a vector $\bp \in E^*$, the Mirror Descent step is defined as

\begin{equation}
\x^+ = \mathrm{Mirr}[\x](\bp):= \arg\min\limits_{\z\in X} \big\{ \la \bp, \z \ra + V[\x](\z) \big\} = \arg\min\limits_{\z\in X} \big\{ \la \bp, \z \ra + d(\z) - \la \nabla d(\x), \z \ra \big\}.
\label{eq:MDStep}
\end{equation}
We make the simplicity assumption, which means that $\mathrm{Mirr}[\x] (\bp)$ is easily computable.

The following lemma \cite{ben-tal2015lectures} describes the main property of the Mirror Descent step.
\begin{lemma}
	\label{Lm:MDProp}
 Let $f$ be some convex function over a set $X$, $h > 0$ be a step-size, $\x \in X^{0}$. Let the point $\x^+$ be defined by
 $ \x^+ = \mathrm{Mirr}[\x](h (\nabla f(\x)))$. Then, for any $\z \in X$,
 \begin{align}
 h \big( f(\x) - f(\z)\big) \leq h \la \nabla f(\x), \x - \z \ra \notag \\
		& \leq \frac{h^2}{2} \| \nabla f (\x)\pd{\|^2} + V[\x](\z) - V[\x^+](\z).	 \label{eq:MDProp}
 \end{align}
\end{lemma}

The following analog of Lemma \ref{Lm:MDProp} for $\delta$-subgradient $\nabla_{\delta} f$ holds.

\begin{lemma}
	\label{Lm:MDPropDelta}
 Let $f$ be some convex function over a set $X$, $h > 0$ be a step-size, $\x \in X^{0}$. Let the point $\x^+$ be defined by
 $ \x^+ = \mathrm{Mirr}[\x](h \cdot (\nabla_{\delta} f(\x)))$. Then, for any $\z \in X$,
 \begin{align}
 h \cdot \big( f(\x) - f(\z)\big) \leq h \cdot \la \nabla f(\x), \x - \z \ra + h \cdot \delta \notag \\
		& \leq \frac{h^2}{2} \| \nabla_{\delta} f (\x) \| + h \cdot \delta + V[\x](\z) - V[\x^+](\z).	 \label{eq:MDPropDelta}
 \end{align}
\end{lemma}

We consider problem \eqref{eq:PrSt} in two different settings, namely, non-smooth Lipschitz-continuous objective function $f$ and general objective function $f$, which is not necessarily Lipschitz-continuous, e.g. a quadratic function. In both cases, we assume that $g$ is non-smooth and is Lipschitz-continuous
\begin{equation}
|g(\x)-g(\y)|\leq M_g\|\x-\y\|_E, \quad \x,\y \in X.
\label{eq:gLipCont}
\end{equation}
Let $\x_*$ be a solution to \eqref{eq:PrSt}. We say that a point $\tilde{\x} \in X$ is an \textit{$\e$-solution} to \eqref{eq:PrSt} if
\begin{equation}
 \label{eq:DetSolDef}
 f(\tilde{\x}) - f(\x_{*}) \leq \e, \quad g(\tilde{\x}) \leq \e.
\end{equation}

All considered in this section methods (Algorithms \ref{Alg:MDNS} and \ref{Alg:MDG}) are applicable in the case of using $\delta$-subgradient instead of usual subgradient. For this case we can get an $\e$-solution $\tilde{\x} \in X$:
\begin{equation}
 \label{eq:DetSolDefDelta}
 f(\tilde{\x}) - f(\x_{*}) \leq \e + O(\delta), \quad g(\tilde{\x}) \leq \e + O(\delta).
\end{equation}
The methods we describe are based on the of Polyak's switching subgradient method \cite{polyak1967general} for constrained convex problems, also analyzed in \cite{nesterov2010introduction}, and Mirror Descent method originated in \cite{nemirovsky1983problem}; see also \cite{nemirovskii1979efficient}.

\subsection{Convex Non-Smooth Objective Function}
\label{S:CNSOF}
In this subsection, we assume that $f$ is a non-smooth Lipschitz-continuous function
\begin{equation}
|f(\x)-f(\y)|\leq M_f\|\x-\y\|_E, \quad \x,\y \in X.
\label{eq:fLipCont}
\end{equation}
Let $\x_*$ be a solution to \eqref{eq:PrSt} and assume that we know a constant $\Theta_0 > 0$ such that
\begin{equation}
d(\x_{*}) \leq \Theta_0^2.
\label{eq:dx*Bound}
\end{equation}
For example, if $X$ is a compact set, one can choose $\Theta_0^2 = \max_{ \x \in X} d(\x)$.

\begin{algorithm}[h!]
\caption{Adaptive Mirror Descent (Non-Smooth Objective)}
\label{Alg:MDNS}
\begin{algorithmic}[1]
 \REQUIRE accuracy $\e > 0$; $\Theta_0$ s.t. $d(\x_{*}) \leq \Theta_0^2$.
 \STATE $\x^0 = \arg\min\limits_{\x\in X} d(\x)$.
	\STATE Initialize the set $I$ as empty set.
	\STATE Set $k=0$.
	\REPEAT
 \IF{$g(\x^k) \leq \e$}
 \STATE $M_k = \| \nabla f(\x^k) \|_{E,*}$,
 \STATE $h_k = \frac{\e}{M_k^2}$
 \STATE $\x^{k+1} =\mathrm{Mirr}[\x^k](h_k \nabla f(\x^k))$ ("productive step")
 \STATE Add $k$ to $I$.
 \ELSE
 \STATE $M_k = \| \nabla g(\x^k) \|_{E,*}$
 \STATE $h_k = \frac{\e}{M_k^2}$
 \STATE $\x^{k+1} = \mathrm{Mirr}[\x^k](h_k \nabla g(\x^k))$ ("non-productive step")
 \ENDIF
 \STATE Set $k = k + 1$.
 \UNTIL{$\sum\limits_{j =0 }^{k-1} \frac{1}{M_j^2} \geq \frac{2\Theta_0^2}{\e^2}$}
 \ENSURE $\bar{\x}^k:= \frac{\sum\limits_{i\in I} h_i \x^i}{\sum\limits_{i\in I} h_i}$
\end{algorithmic}
\end{algorithm}

\begin{theorem}
	\label{Th:MDCompl}
	Assume that inequalities \eqref{eq:gLipCont} and \eqref{eq:fLipCont} hold and a known constant $\Theta_0 > 0$ is such that $
d(\x_{*}) \leq \Theta_0^2$. Then, Algorithm \ref{Alg:MDNS} stops after not more than
	\begin{equation}
	k = \left\lceil\frac{2\max\{M_f^2,M_g^2\} \Theta_0^2}{\e^2}\right\rceil
	\label{eq:MDNSComplEst}
	\end{equation}
	iterations and $\bar{\x}^k$ is an $\e$-solution to \eqref{eq:PrSt} in the sense of \eqref{eq:DetSolDef}.
\end{theorem}

Let us now show that Algorithm \ref{Alg:MDNS} allows to reconstruct an approximate solution to the problem, which is dual to \eqref{eq:PrSt}. We consider a special type of problem \eqref{eq:PrSt} with $g$ given by
\begin{equation}
g(\x) = \max\limits_{i \in \{1,..., m\}} \big\{ g_i(\x) \big\}.
\label{eq:gMaxDef}
\end{equation}
Then, the dual problem to \eqref{eq:PrSt} is
\begin{equation}\label{eq:dualPrSt}
 \varphi (\lambda) = \min\limits_{\x\in X} \Big\{ f(\x) + \sum\limits_{i=1}^{m} \lambda_i g_i(\x) \Big\} \rightarrow \max\limits_{\lambda_i \geq 0, i = 1,..., m}\varphi (\lambda),
\end{equation}
where $\lambda_i \geq 0, i = 1,..., m$ are Lagrange multipliers.

We slightly modify the assumption \eqref{eq:dx*Bound} and assume that the set $X$ is bounded and that we know a constant $\Theta_0 > 0$ such that
$$
\max_{\x \in X} d(\x) \leq \Theta_0^2.
$$

As before, denote $[k] = \{ j \in \{0,...,k-1\} \}$, $J = [k] \setminus I$.
Let $j \in J$. Then a subgradient of $g(\x)$ is used to make the $j$-th step of Algorithm \ref{Alg:MDNS}. To find this subgradient, it is natural to find an active constraint $i \in 1,..., m$ such that $g(\x^j) = g_{i}(\x^j)$ and use $\nabla g(\x^j) = \nabla g_{i}(\x^j)$ to make a step. Denote $i(j) \in 1,..., m$ the number of active constraint, whose subgradient is used to make a non-productive step at iteration $j \in J$. In other words, $g(\x^j) = g_{i(j)}(\x^j)$ and $\nabla g(\x^j) = \nabla g_{i(j)}(\x^j)$.
We define an approximate dual solution on a step $k\geq 0$ as
 \begin{equation}\label{eq:lamDef}
 \bar{\lambda}_i^k = \frac{1}{\sum\limits_{j\in I} h_j} \sum\limits_{j\in J, i(j) = i} h_j, \quad i \in \{1,..., m\}.
 \end{equation}
		and modify Algorithm \ref{Alg:MDNS} to return a pair $(\bar{\x}^k,\bar{\lambda}^k)$.

\begin{theorem}
\label{Th:MDPDCompl}
 Assume that the set $X$ is bounded, the inequalities \eqref{eq:gLipCont} and \eqref{eq:fLipCont} hold and a known constant $\Theta_0 > 0$ is such that $
d(\x_{*}) \leq \Theta_0^2$.
 Then, modified Algorithm \ref{Alg:MDNS} stops after not more than
	\begin{equation}
	k = \left\lceil\frac{2\max\{M_f^2,M_g^2\} \Theta_0^2}{\e^2}\right\rceil \notag
		\end{equation}
	iterations and the pair $(\bar{\x}^k,\bar{\lambda}^k)$ returned by this algorithm satisfies
 \begin{equation}
		\label{eq:DetPDSolDef}
 f(\bar{\x}^k) - \varphi(\bar{\lambda}^k) \leq \e, \quad g(\bar{\x}^k) \leq \e.
 \end{equation}
\end{theorem}

Now we consider an interesting example of huge-scale problem
\cite{nesterov2014subgradient, nesterov2014primal-dual} with a sparse structure.
We would like to illustrate two important ideas. Firstly, consideration of the dual problem can simplify the solution, if it is possible to reconstruct the solution of the primal problem by solving the dual problem. Secondly, for a special sparse non-smooth piece-wise linear functions we suggest a very efficient implementation of one subgradient iteration \cite{nesterov2014subgradient}. In such cases simple subgradient methods (for example, Algorithm \ref{Alg:MDNS}) can be useful due to the relatively inexpensive cost of iterations.

Recall (see e.g. \cite{nesterov2014primal-dual}) that Truss Topology Design problem consists in finding the best mechanical structure resisting
to an external force with an upper bound for the total weight of construction. Its
mathematical formulation looks as follows:
\begin{equation}\label{4.1}
\min_{w\in R^{m}_{+}}\{\langle \overline{\f},\z\rangle:\,A(\bw)\z=\overline{\f},\,\langle \be,\bw\rangle=T\},
\end{equation}
where $\overline{\f}$ is a vector of external forces, $\z\in R^{2n}$ is a vector of virtual displacements of $n$ nodes in $R^{2}$, $\bw$ is a vector of $m$ bars, and $T$ is the total weight of construction. The compliance matrix $A(\bw)$ has the following form:
$$A(\bw)=\sum_{i=1}^{m}\bw_{i}\ba_{i}\ba_{i}^{T},$$
where $\ba_{i}\in R^{2n}$ are the vectors describing the interactions of two nodes connected by an arc. These vectors are very sparse: for 2D-model they have at most 4 nonzero elements.

Let us rewrite the problem \eqref{4.1} as a Linear Programming problem.
\begin{equation}\label{4.2}
\begin{split}
&\min_{\z,\bw}\{\langle \overline{\f},\z\rangle:\,A(\bw)\z=\overline{\f},\,\bw\geq0,\,\langle \be,\bw\rangle=T\}=\\
&=\min_{\bw}\{\langle \overline{\f},A^{-1}(\bw)\overline{\f}\rangle:\,\bw\in \triangle(T)=\{\,\bw\geq0,\,\langle \be,\bw\rangle=T\}\}=\\
&=\min_{\bw\in \triangle(T)}\max_{\z}\{2\langle\overline{\f},\z\rangle-\langle A(\bw)\z,\z\rangle\}\geq\max_{\z}\min_{w\in \triangle(T)}\{2\langle\overline{\f},\z\rangle-\langle A(w)\z,\z\rangle\}=\\
&=\max_{\z}\{2\langle\overline{\f},\z\rangle- T\max_{1\leq i\leq m}\langle \ba_{i},\z\rangle^{2}\}=\max_{\lambda,\y}\{2\lambda\langle\overline{\f},\y\rangle- \lambda^{2}T\max_{1\leq i\leq m}\langle \ba_{i},\y\rangle^{2}\}=\\
&=\max_{\y}\frac{\langle\overline{\f},\y\rangle^{2}}{T\max\limits_{1\leq i\leq m}\langle \ba_{i}, \y\rangle^{2}}=\frac{1}{T}\left(\max_{\y}\{\langle \overline{\f},\y\rangle:\,\max_{1\leq i\leq m}|\langle \ba_{i}, \y\rangle|\leq1\}\right)^{2}.
\end{split}
\end{equation}
Note that for the inequality in the third line we do not need any assumption.

Denote by $\y^{*}$ the optimal solution of the optimization problem in the brackets. Then
there exist multipliers $\x^{*}\in R^{m}_{+}$ such that
\begin{equation}\label{4.3}
\overline{\f}=\sum_{i\in J_{+}}\ba_{i}\x_{i}^{*}-\sum_{i\in J_{-}}\ba_{i} \x_{i}^{*},\qquad \x_{i}^{*}=0,\, i\not\in J_{+}\bigcap J_{-},
\end{equation}
where $J_{+}=\{i:\,\langle \ba_{i},\y^{*}\rangle=1\}$, and $J_{-}=\{i:\,\langle \ba_{i},\y^{*}\rangle=-1\}$. Multiplying the first equation
in \eqref{4.3} by $\y^{*}$, we get
\begin{equation}\label{4.4}
\langle \overline{\f},\y^{*}\rangle=\langle \be,\x^{*}\rangle.
\end{equation}
Note that the first equation in \eqref{4.3} can be written as
\begin{equation}\label{4.5}
\overline{\f}=A(\x^{*})\y^{*}.
\end{equation}
Let us reconstruct now the solution of the primal problem. Denote
\begin{equation}\label{4.6}
w^{*}=\frac{T}{\langle \be,\x^{*}\rangle}\cdot \x^{*},\qquad \z^{*}=\frac{\langle \be,\x^{*}\rangle}{T}\cdot \y^{*}.
\end{equation}
Then, in view of \eqref{4.5} we have $\overline{\f}=A(\bw^{*})\z^{*}$, and $\bw^{*}\in \triangle(T)$. Thus, the pair \eqref{4.6} is
feasible for the primal problem. On the other hand,
$$\langle \overline{\f},\z^{*}\rangle=\langle \overline{\f},\frac{\langle \be,\x^{*}\rangle}{T}\cdot \y^{*}\rangle =\frac{1}{T}\cdot\langle \be,\x^{*}\rangle\cdot\langle \overline{\f},\y^{*}\rangle=\frac{1}{T}\cdot\langle \overline{\f},\y^{*}\rangle^{2}.$$
Thus, the duality gap in the chain \eqref{4.2} is zero, and the pair $(\bw^{*}, \z^{*})$, defined by \eqref{4.6} is the optimal solution of the primal problem.

The above discussion allows us to concentrate on the following (dual) Linear Programming
problem:
\begin{equation}\label{4.7}
\max_{\y}\{ \iprod{\bar \f}{\y}:\,\max_{1 \leq i \leq m} \iprod{\pm \ba_{i}} {\y} \leq 1\},
\end{equation}
which we can solve by the primal-dual Algorithm \ref{Alg:MDNS}.

Assume that we have {\it local}
truss: each node is connected only with few neighbors.
It allows to apply the property of {\it sparsity} for vectors $\ba_{i}$ ($1\leq i\leq m$).
In this case the computational cost of each iteration grows as $O(\log_{2} m)$ \cite{nesterov2014subgradient, nesterov2014primal-dual}.

In \cite{nesterov2014subgradient} a special class of huge-scale problems with sparse subgradient was considered.
According to \cite{nesterov2014subgradient} for smooth functions this is a very rare feature.
For example, for quadratic function $f(\y) = \frac{1}{2} \iprod{A\y}{\y} $
the gradient $\nabla f(\y) = A\y$ usually is dense even for a sparse matrix $A$.

However, the subgradient of non-smooth function
$f(\y) = \max_{1\leq i\leq m} \iprod{\ba_{i}}{\y} $
(see \eqref{4.7} above) are sparse provided that all vectors $\ba_i$ share this property.
This fact is based on the following observation.
For the function $f(\y) = \max_{1\leq i\leq m}\iprod{\ba_{i}}{\y}$
with sparse matrix $A = (\ba_{1}, \ba_{2},..., \ba_{m})$
the vector $\nabla f(\y) = \ba_{i(\y)}$ is a subgradient at point $\y$.
Then the standard subgradient step
$$\y_{+} = \y - h \cdot \nabla f(\y)$$
changes only a few entries of vector $\y$ and the vector $\z_+ = A^{T} \y_+$ differs from
$\z = A^{T}\y$ also in a few positions only.
Thus, the function value $f(\y_+)$ can be easily updated provided that we have an efficient procedure
for recomputing the maximum of $m$ values.

Note the objective functional in \eqref{4.7} is linear and the costs of iteration of Algorithm \ref{Alg:MDNS}
and considered in \cite{nesterov2014subgradient} switching simple subgradient scheme is comparable.
At the same time, the step productivity condition is simpler for Algorithm \ref{Alg:MDNS}
as considered in \cite{nesterov2014subgradient} switching subgradient scheme.
Therefore main observations for \cite{nesterov2014subgradient} are correct for Algorithm \ref{Alg:MDNS}.

\subsection{General Convex and Quasi-Convex Objective Functions}
\label{GCandQCobj}

In this subsection, we assume that the objective function $f$ in
\eqref{eq:PrSt} might not satisfy \eqref{eq:fLipCont} and,
hence, its subgradient could be unbounded.
One of the examples is a quadratic function.
We also assume that inequality \eqref{eq:dx*Bound} holds.

We further consider ideas in \cite{nesterov2010introduction,nesterov2015subgradient} and adapt
them for problem \eqref{eq:PrSt}, in a way that our algorithm allows to use non-Euclidean proximal setup,
as does Mirror Descent, and does not require to know the constant $M_g$.
Following \cite{nesterov2010introduction},
given a function $f$ for each subgradient $\nabla f(\x)$ at a point $\y \in X$,
we define
\begin{equation}
v_f[\y](\x)=\left\{
\begin{aligned}
&\left\la\frac{\nabla f(\x)}{\|\nabla f(\x)\|_{E,*}},\x-\y\right\ra, \quad &\nabla f(\x) \ne 0\\
&0 &\nabla f(\x) = 0\\
\end{aligned}
\right.,\quad \x \in X.
\label{eq:vfDef}
\end{equation}

\begin{algorithm}[h!]
\caption{Adaptive Mirror Descent (General Convex Objective)}
\label{Alg:MDG}
\begin{algorithmic}[1]
 \REQUIRE accuracy $\e > 0$; $\Theta_0$ s.t. $d(\x_{*}) \leq \Theta_0^2$.
 \STATE $\x^0 = \arg\min\limits_{\x\in X} d(\x)$.
	\STATE Initialize the set $I$ as empty set.
	\STATE Set $k=0$.
	\REPEAT
 \IF{$g(\x^k) \leq \e$}
 \STATE $h_k = \frac{\e}{\| \nabla f(\x^k) \|_{E,*}}$
 \STATE $\x^{k+1} =\mathrm{Mirr}[\x^k](h_k \nabla f(\x^k))$ ("productive step")
 \STATE Add $k$ to $I$.
 \ELSE
 \STATE $h_k = \frac{\e}{\| \nabla g(\x^k) \|_{E,*}^2}$
 \STATE $\x^{k+1} = \mathrm{Mirr}[\x^k](h_k \nabla g(\x^k))$ ("non-productive step")
 \ENDIF
 \STATE Set $k = k + 1$.
 \UNTIL{$|I|+\sum\limits_{j \in J} \frac{1}{\| \nabla g(\x^j) \|_{E,*}^2} \geq \frac{2\Theta_0^2}{\e^2} $}
 \ENSURE $\bar{\x}^k:= \arg \min_{\x^j, j\in I} f(\x^j)$
\end{algorithmic}
\end{algorithm}

The following result gives complexity estimate for Algorithm \ref{Alg:MDG} in terms of $v_f[\x_*](\x)$.
Below we use this theorem to establish complexity result for smooth objective $f$.
\begin{theorem}
	\label{Th:MDGCompl}
	Assume that inequality \eqref{eq:gLipCont} holds and a known constant $\Theta_0 > 0$ is such that $
d(\x_{*}) \leq \Theta_0^2$. Then, Algorithm \ref{Alg:MDG} stops after not more than
	\begin{equation}
	k = \left\lceil\frac{2\max\{1,M_g^2\} \Theta_0^2}{\e^2}\right\rceil
	\label{eq:MDGComplEst}
	\end{equation}
	iterations and it holds that
	$\min_{i \in I} v_f[\x_*](\x^i) \leq \e$ and $g(\bar{\x}^k)\leq \e$.
\end{theorem}

To obtain the complexity of our algorithm in terms of the values of the objective function $f$, we define non-decreasing function
\begin{equation}
\omega(\tau)=
\left\{
\begin{aligned}
&\max\limits_{\x\in X}\{f(\x)-f(\x_*):\|\x-\x_*\|_{E}\leq \tau\} \quad &\tau \geq 0,\\
&0 & \tau <0.\\
\end{aligned}
\right.
\label{eq:omDef}
\end{equation}
and use the following lemma from \cite{nesterov2010introduction}.
\begin{lemma}
\label{Lm:fGrowthOm}
Assume that $f$ is a convex function. Then, for any $\x \in X$,
\begin{equation}
f(\x) - f(\x_*) \leqslant \omega(v_f[\x_*](\x)).
\label{eq:fGrowth}
\end{equation}
\end{lemma}

\begin{corollary}
\label{Col:MDSCompl}
Assume that the objective function $f$ in \eqref{eq:PrSt} is given as
$f(\x) = \max_{i \in \{1,...,m\}} f_i(\x)$, where $f_i(\x)$, $i = 1,...,m$
are differentiable with Lipschitz-continuous gradient
\begin{equation}
\|\nabla f_i(\x)-\nabla f_i(\y)\|_{E,*} \leq L_i\|\x-\y\|_{E} \quad \forall \x,\y\in X, \quad i \in \{1,...,m\}.
\label{eq:fLipSm}
\end{equation}
Then $\bar{\x}^k$ is $\widetilde{\varepsilon}$-solution to \eqref{eq:PrSt}
in the sense of \eqref{eq:DetSolDef},
where
$$\widetilde{\varepsilon} = \max\{\e, \e \max_{i = 1,...,m}\|\nabla f_i(\x_*)\|_{E,*}+\e^2\max_{i = 1,...,m}L_i/2\}.$$
\end{corollary}

\begin{remark}
According to \cite{nesterov1989, nesterov2018lectures} main lemma \ref{Lm:fGrowthOm} holds for quasi-convex objective functions
\cite{Polyak1969} too:
\[
f\left( {\alpha \x+\left( {1-\alpha } \right)\y} \right)\le \max \left\{
{f\left( \x \right),f\left( \y \right)} \right\} \text{ for all }\x, \y, \alpha \in
[0,1].
\]
This means that results of this subsection are valid for quasi-convex objectives.
\end{remark}

\begin{remark}
In view of the Lipschitzness and, generally speaking, non-smoothness of functional limitations,
the obtained estimate for the number of iterations means
that the proposed method is optimal from the point of view of oracle evaluations:
$O\left(\frac{1}{\varepsilon^2}\right)$ iterations are sufficient for
achieving the required accuracy $\varepsilon$ of solving the problem for
the class of target functionals considered in this section of the article.
Note also that the considered algorithm \ref{Alg:MDNS} applies to the considered classes of problems
with constraints with convex objective functionals of different smoothness levels.
However, the non-fulfillment, generally speaking, of the Lipschitz condition
for the objective functional $f$ does not allow one to substantiate
the optimality of the algorithms \ref{Alg:MDNS} in the general situation
(for example, with a Lipschitz-continuous gradient).
More precisely, situations are possible when the productive steps
of the norm (sub)gradients of the objective functional $\|\nabla f(\x^k)\|_*$ are large enough
and this will interfere with the speedy achievement of the stopping criterion of the \ref{Alg:MDNS}.
\end{remark}
\section{Universal Methods}
\label{unimeth}
In this section we consider problem
\begin{equation}
 \label{eq:UM-Problem}
 \min_{\x\in Q \subseteq E} f(\x),
\end{equation}
where $Q$ is a convex set and $f$ is a convex function with H\"older-continuous subgradient
\begin{equation}
 \label{eq:UM-Hoeld-Cond}
 \|\nabla f(\x_1) - \nabla f(\x_2)\|_* \leq L_{\nu}\|\x_1-\x_2\|^{\nu}
\end{equation}
with $\nu \in [0,1]$. The case $\nu=0$ corresponds to non-smooth optimization and the case $\nu=1$ corresponds to smooth optimization. The goal of this section is to present the Universal Accelerated Gradient method first proposed by Nesterov \cite{nesterov2015universal}. This method is a black-box method which does not require the knowledge of constants $\nu,L_{\nu}$ and works in accordance with the lower complexity bound $O\left(\left(\frac{L_{\nu}R^{1+\nu}}{\ee}\right)^{\frac{2}{1+3\nu}}\right)$ obtained in \cite{nemirovsky1983problem}.

The main idea is based on the observation that a non-smooth convex function can be upper bounded by a quadratic objective function slightly shifted above. More precisely, for any $\x,\y \in Q$,
\begin{align}
 f(\y) &\leq f(\x) + \iprod{\nabla f(\x)}{\y-\x} + \frac{L_{\nu}}{1+\nu}\|\y-\x\|^{1+\nu}\notag \\
 &\leq f(\x) + \iprod{\nabla f(\x)}{\y-\x} + \frac{L(\delta)}{2}\|\y-\x\|^{2} + \delta, \label{eq:UM-Upper-Bound}
\end{align}
where
\[
L(\delta) = \left(\frac{1-\nu}{1+\nu}\frac{1}{\delta} \right)^{\frac{1-\nu}{1+\nu}} L_{\nu}^{\frac{2}{1+\nu}}.
\]

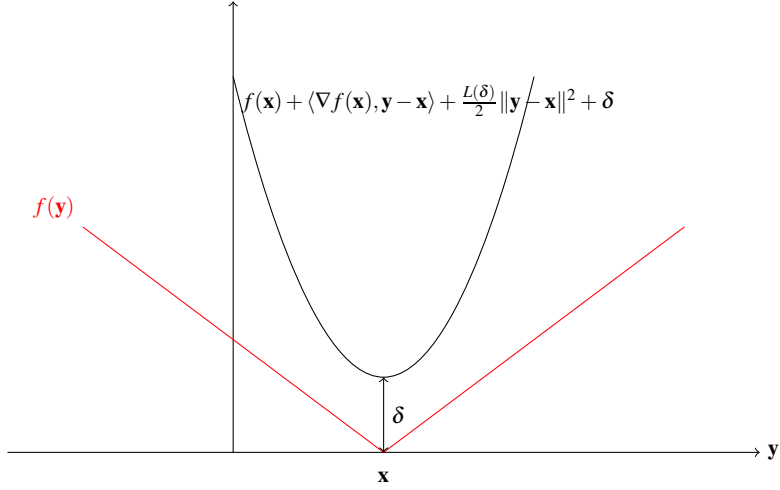
\begin{figure}
\centering
\begin{tikzpicture}[scale=1]
 \coordinate (y) at (-2,6);
 \coordinate (x) at (5,0);
 \draw[<->] (y) node[above] {} -- (-2,0) -- (x) node[right]
 {$\y$};
	\draw[-] (-5,0) -- (-2,0);

 \draw (0,-0.5) node[above] {$\x$};

 \draw (-2,5) node[below right] {$f(\x) + \iprod{\nabla f(\x)}{\y-\x} + \frac{L(\delta)}{2}\|\y-\x\|^{2} + \delta$} parabola bend (0,1) (2,5);
		
		\draw[red,-] (-4,3) node[above left] {{\color{red}$f(\y)$}} -- (0,0);
		\draw[red,-] (4,3) -- (0,0);
		
	\draw[<->] (0,0) -- (0,1);
	
	\draw (0,0.5) node[right] {$\delta$};
		
 \end{tikzpicture}
 \caption{Quadratic majorant of a non-smooth function $f(\x)$.}
\end{figure}

The next idea is to apply an accelerated gradient method with backtracking procedure to adapt for the unknown $L(\delta)$ with appropriately chosen $\delta$. The method we present is based on accelerated gradient method from \cite{dvurechensky2017adaptive,dvurechensky2018computational} and, thus is different from the original method of \cite{nesterov2015universal}.

\begin{algorithm}[h!]
\caption{Universal Accelerated Gradient Method}
\label{Alg:UM-UAGM}
{\small
\begin{algorithmic}[1]
		\REQUIRE Accuracy $\ee$, starting point $\x^0 \in Q$, initial guess $L_0 >0$, prox-setup: $d(\x)$ -- $1$-strongly convex w.r.t. $\|\cdot\|_{E}$, $V[\z] (\x):= d(\x) - d(\z) - \la \nabla d(\z), \x - \z \ra$.
		\STATE Set $k=0$, $C_0=\alpha_0=0$, $\y^0=\z^0=\x^0$.
		\FOR{$k=0,1,...$}
			\STATE Set $M_k=L_k/2$.
			\REPEAT
				\STATE Set $M_k=2M_k$, find $\alpha_{k+1}$ as the largest root of the equation
				\begin{equation}
				C_{k+1}:=C_k+\alpha_{k+1} = M_k\alpha_{k+1}^2.
				\label{eq:UM-alpQuadEq}
				\end{equation}
				\STATE
				\begin{equation}
				\x^{k+1} = \frac{\alpha_{k+1}\z^{k} + C_k \y^{k}}{C_{k+1}}.
				\label{eq:UM-lambdakp1Def}
				\end{equation}
				\STATE
				\begin{equation}
				\z^{k+1}=\arg \min_{\x \in Q} \{V[\z^{k}](\x) + \alpha_{k+1}(f(\x^{k+1}) + \langle \nabla f(\x^{k+1}), \x - \x^{k+1} \rangle) \}.
				\label{eq:UM-zetakp1Def}
				\end{equation}
				\STATE
				\begin{equation}
				\y^{k+1} = \frac{\alpha_{k+1}\z^{k+1} + C_k \y^{k}}{C_{k+1}}.
				\label{eq:UM-etakp1Def}
				\end{equation}
			\UNTIL{
			\begin{equation}
			f(\y^{k+1}) \leq f(\x^{k+1}) + \la \nabla f(\x^{k+1}),\y^{k+1} - \x^{k+1} \ra +\frac{M_k}{2}\|\y^{k+1} - \x^{k+1}\|^2 + \frac{\alpha_{k+1} \ee}{2C_{k+1}}.
			\label{eq:UM-lipConstCheck}
			\end{equation}}
			\STATE Set $L_{k+1}=M_k/2$, $k=k+1$.
		\ENDFOR
		\ENSURE The point $\y^{k+1}$.	
\end{algorithmic}
}
\end{algorithm}

Inequality \eqref{eq:UM-Upper-Bound} guarantees that the backtracking procedure in the inner cycle is finite.
\begin{theorem}[\cite{nesterov2015universal}]
\label{Th:UM-Main}
Let $f$ satisfy \eqref{eq:UM-Hoeld-Cond}. Then,
\begin{equation}
 \label{eq:UM-Conv-Rate}
 f(\y^{k+1})-f_\star \leq \left(\frac{2^{2+4\nu}L_{\nu}^2}{\ee^{1-\nu}k^{1+3\nu}}\right)^{\frac{1}{1+\nu}} V[\x^0](\x^\star) + \frac{\ee}{2}.
\end{equation}
Moreover, the number of oracle calls is bounded by
\[
4(k+1) + 2 \log_2 \left((2V[\x^0](\x^\star))^{\frac{1-\nu}{1+3\nu}} \left(\frac{1}{\ee}\right)^{\frac{3(1-\nu)}{1+3\nu}}L_{\nu}^{\frac{4}{1+3\nu}} \right).
\]
\end{theorem}
Translating this rate of convergence to the language of complexity, we obtain that to obtain a solution with an accuracy $\ee$ the number of iterations is no more than
\[
O\left( \inf_{\nu \in [0,1]} \left(\frac{L_{\nu}}{\ee}\right)^{\frac{2}{1+3\nu}} \left(V[\x^0](\x^\star)\right)^{\frac{1+\nu}{1+3\nu}} \right),
\]
i.e. is optimal.

In his paper, Nesterov considers a more general composite optimization problem
\begin{equation}
 \label{eq:UM-Problem-Comp}
 \min_{\x\in Q \subseteq E} f(\x) + h(\x),
\end{equation}
where $h$ is a simple convex function, and obtains the same complexity guarantees. Universal methods were extended for the case of strongly convex problems by a restart technique in \cite{roulet2017sharpness}, for non-convex optimization in \cite{ghadimi2015generalized} and for the case of non-convex optimization with inexact oracle in \cite{dvurechensky2017gradient}.
As we can see from \eqref{eq:UM-Upper-Bound}, universal accelerated gradient method is connected to smooth problems with inexact oracle. The study of accelerated gradient methods with inexact oracle was first proposed in \cite{aspremont2008smooth} and was very well developed in \cite{devolder2014first,dvurechensky2016stochastic,bogolubsky2016learning,dvurechensky2017gradient} including stochastic optimization problems and strongly convex problems. A universal method with inexact oracle can be found in \cite{dvurechensky2017universal}.
Experiments show \cite{nesterov2015universal} that universal method accelerates to $O\left( \frac{1}{k}\right)$ rate for non-smooth problems with a special "smoothing friendly" (see Section \ref{black-box}) structure. This is especially interesting for traffic flow modeling problems, which possess such structure \cite{baimurzina2017universal}.


Now we consider universal analog of A.S. Nemirovski\rq s proximal mirror method
for variational inequalities with a Holder-continuous operator. More precisely, we consider universal extension of Algorithm \ref{Alg:Sm-ng-UMP} which allows to solve smooth and non-smooth variational inequalities without the prior knowledge of the smoothness. Main idea of the this method is the adaptive choice of constants and level of smoothness in minimized prox-mappings at each iteration.
These constants are related to the H\"older constant of the operator and this method allows to find a suitable constant at each iteration.

\begin{algorithm}[ht]
\caption{Universal Mirror Prox}
\label{Alg:UMP}
\begin{algorithmic}[1]
 \REQUIRE General VI on a set $Q \subset E$ with operator $\Phi(\z)$, accuracy $\e > 0$, initial guess $M_{-1} >0$, prox-setup: $d(\z)$, $V[\z] (\bw)$.
 \STATE Set $k=0$, $\z^0 = \arg \min_{\z \in Q} d(\z)$.
 \FOR{$k=0,1,...$}
 \STATE Set $i_k=0$
			\REPEAT
				\STATE Set $M_k=2^{i_k-1}M_{k-1}$.
		\STATE Calculate
				\begin{equation}
				\bw^k={\mathop {\arg \min }\limits_{\z\in Q}}\left\{\la \Phi(\z^k),\z \ra + M_{k}V[\z^k](\z) 		 \right\}.
				\label{eq:UMPwStep}
				\end{equation}
				\STATE Calculate
				\begin{equation}
				\z^{k+1}={\mathop {\arg \min }\limits_{\z\in Q}} \left\{\la \Phi(\bw^k),\z \ra + M_{k}V[\z^k](\z) 		 \right\}.
				\label{eq:UMPzStep}
				\end{equation}

	\STATE $i_k=i_k+1$.
			\UNTIL{
			\begin{equation}
			\la \Phi(\bw^k) - \Phi(\z^k), \bw^k - \z^{k+1} \ra
 \leq \frac{M_k}{2}\left(\|\bw^k-\z^k\|^2 + \|\bw^k-\z^{k+1}\|^2\right) + \frac{\e}{2}.
			\label{eq:UMPCheck}
			\end{equation}}
					\STATE Set $k=k+1$.
	\ENDFOR
		\ENSURE $\widehat{\bw}^k = \frac{1}{k}\sum_{i=0}^{k-1}\bw^i$.
\end{algorithmic}
\end{algorithm}

\newpage

\begin{theorem}[\cite{dvurechensky2018generalized}]
\label{Th:UMPGenRate}
For any $k \geq 1$ and any $\z \in Q$,
\begin{equation}
\label{eq:UMPGenRate}
\frac{1}{\sum_{i=0}^{k-1}M_i^{-1}} \sum_{i=0}^{k-1} M_i^{-1} \la \Phi(\bw^i), \bw^i - \z \ra \leq
 \frac{1}{\sum_{i=0}^{k-1}M_i^{-1}} (V[\z^0](\z)-V[\z^{k}](\z)) + \frac{\e}{2}.
\end{equation}
\end{theorem}

Note that if $\max_{\z \in Q} V[\z^0](\z) \leq D$, we can construct the following adaptive stopping criterion for our algorithm
$$
\frac{D}{\sum_{i=0}^{k-1}M_i^{-1}} \leq \frac{\e}{2}.
$$

Next, we consider the case of H\"older-continuous operator $\Phi$ and
show that Algorithm \ref{Alg:UMP} is universal.
Assume for some $\nu \in [0,1]$ and $L_{\nu} \geq 0$
$$
\|\Phi(\x) - \Phi(\y)\|_* \leq L_{\nu}\|\x-\y\|^{\nu}, \quad \x,\y \in Q.
$$
holds.
The following inequality is a generalization of \eqref{eq:UM-Upper-Bound} for VI. For any $\x,\y,\z \in Q$ and $\delta > 0$,
$$
	\la \Phi(\y) - \Phi(\x), \y-\z \ra \leq \|\Phi(\y) - \Phi(\x) \|_* \|\y-\z\| \leq L_{\nu}\|\x-\y\|^{\nu} \|\y-\z\| \leq
$$
$$
\leq \frac{1}{2}\left(\frac{1}{\delta}\right)^{\frac{1-\nu}{1+\nu}} L_{\nu}^{\frac{2}{1+\nu}} \left(\|\x-\y\|^2+\|\y-\z\|^2\right) + \frac{\delta}{2},
$$
where
\begin{equation}
\label{eq:Lofd}
L(\delta) = \left(\frac{1}{\delta}\right)^{\frac{1-\nu}{1+\nu}} L_{\nu}^{\frac{2}{1+\nu}}.
\end{equation}
So, we have
\begin{equation}
	\label{eq:g_or_def}
	\la \Phi(\y) - \Phi(\x), \y - \z \ra \leq \frac{L(\delta)}{2}\left(\|\y-\x\|^2 + \|\y-\z\|^2\right) + \delta.
\end{equation}

Let us consider estimates of the necessary number of iterations are obtained to achieve a given quality of the variational inequality solution.

\begin{corollary}[Universal Method for VI]
\label{Cor:UMPHoldUniv}
Assume that the operator $\Phi$ is H\"older continuous with constant $L_{\nu}$ for some $\nu \in [0,1]$ and $M_{-1} \leq \left(\frac{2}{\e}\right)^{\frac{1-\nu}{1+\nu}} L_{\nu}^{\frac{2}{1+\nu}}$. Also assume that the set $Q$ is bounded. Then, for all $k \geq 0$, we have
\begin{equation}
\label{eq:UMPRate}
\max_{\z\in Q} \la \Phi(\z), \widehat{\bw}_k - \z \ra \leq \frac{(2 L_{\nu})^{\frac{2}{1+\nu}}}{ k \e^{\frac{1-\nu}{1+\nu}}} \max_{\z\in Q}V[\z^0](\z) + \frac{\e}{2}
\end{equation}
\end{corollary}

As it follows from \eqref{eq:g_or_def}, if $M_k \geq L(\frac{\e}{2})$, \eqref{eq:UMPCheck} holds. Thus, for all $i = 0,..., k-1$, we have $M_i \leq 2\cdot L(\frac{\e}{2})$ and
$$
\frac{1}{\sum_{i=0}^{k-1}M_i^{-1}} \leq \frac{2L(\frac{\e}{2})}{k} \leq \frac{(2 L_{\nu})^{\frac{2}{1+\nu}}}{ k\e^{\frac{1-\nu}{1+\nu}}},
$$
\eqref{eq:UMPRate} holds. Here $L(\cdot)$ is defined in \eqref{eq:Lofd}.
\qed

Let us add some remarks.

\begin{remark}
Since the algorithm does not use the values of parameters $\nu$ and $L_{\nu}$, we obtain the following iteration complexity bound
$$
2 \inf_{\nu\in[0,1]}\left(\frac{2L_{\nu}}{\e} \right)^{\frac{2}{1+\nu}} \cdot \max_{\z\in Q}V[\z_0](\z)
$$
to achieve $$\max_{\z\in Q} \la \Phi(\z), \widehat{\bw}_k - \z \ra \leq \e.$$
\end{remark}

Using the same reasoning as in \cite{nesterov2015universal}, we estimate the number of oracle calls for Algorithm \ref{Alg:UMP}. 
The number of oracle calls on each iteration $k$ is equal to $2i_k$. At the same time, $M_k=2^{i_k-2}M_{k-1}$ and, hence, $i_k = 2+ \log_2\frac{M_k}{M_{k-1}}$.
Thus, the total number of oracle calls is
\begin{equation}\label{equiv_iter_2}
\sum\limits_{j = 0}^{k-1} i_{j} = 4k + 2\sum\limits_{i = 0}^{k-1}\log_{2} \frac{M_{j}}{M_{j-1}} < 4k + 2 \log_{2} \left(2L\left(\frac{\e}{2}\right)\right) - 2 \log_{2} (M_{-1}),
\end{equation}
where we used that $M_{k} \leq 2 L(\frac{\e}{2})$.

Thus, the number of oracle calls of the Algorithm \ref{Alg:UMP} does not exceed:
$$
4 \inf_{\nu\in[0,1]}\left(\frac{2 \cdot L_{\nu}}{\e} \right)^{\frac{2}{1+\nu}} \cdot \max_{u\in Q}V[z_0](u) + 2 \inf_{\nu\in[0,1]} \log_2 2\left(\left(\frac{2}{\e}\right)^{\frac{1-\nu}{1+\nu}} L_{\nu}^{\frac{2}{1+\nu}}\right) - 2\log_2(M_{-1}).
$$

\begin{remark}
We can apply this method to convex-concave saddle problems of the form
\begin{equation}\label{eqv_sedlo}
f(\x, \y) \rightarrow \min\limits_{\x \in Q_1} \max\limits_{\y \in Q_2},
\end{equation}
where $Q_{1, 2}$ are convex compacts in $\mathbb{R}^n$, $f$ is convex in $\x$ and concave in $\y$,
there is $\nu\in[0,1]$ and constants $L_{11,\nu},L_{12,\nu},L_{21,\nu},L_{22,\nu} \textless +\infty$:
$$
\|\nabla_\x f(\x+\Delta \x, \y+\Delta \y) - \nabla_\x f(\x,\y)\|_{1,*}\leq L_{11,\nu}\| \Delta \x\|_{1}^{\nu}+L_{12,\nu}\| \Delta \y\|_{2}^{\nu},
$$
$$
\|\nabla_\y f(\x+\Delta \x, \y+\Delta \y) - \nabla_\y f(\x,\y)\|_{2,*}\leq L_{21,\nu}\| \Delta \x\|_{1}^{\nu}+L_{22,\nu}\| \Delta \y\|_{2}^{\nu}
$$
for all $\x, \x+\Delta \x\in Q_1, \y, \y+\Delta \y \in Q_2$.

It is possible to achieve an acceptable approximation
$(\widehat{\x}, \widehat{\y}) \in Q_1 \times Q_2$:
\begin{equation}\label{eqv_sedlo_qual}
\max_{\y\in Q_2} f(\widehat{\x},\y) - \min_{\x \in Q_1} f(\x,\widehat{\y}) \leq \varepsilon
\end{equation}
for the saddle point $(\x_*,\y_*) \in Q_1 \times Q_2$ of the \eqref{eqv_sedlo} problem in no more than
$$O\left(\left(\frac{1}{\varepsilon}\right)^{\frac{2}{1+\nu}}\right)$$
iterations, which indicates the optimality of the proposed method, at least for $\nu = 0$ and $\nu = 1$.
However, in practice experiments show that \eqref{eqv_sedlo_qual} can be achieved much faster due to the adaptability of the method.
\end{remark}
\section{Concluding remarks}
\label{conclusion}
Modern numerical methods for non-smooth convex optimization problems are typically based on the structure of the problem.
We start with one of the most powerful example of such type.
For geometric median search problem there exists efficient method
that significantly outperform described above lower complexity bounds \cite{cohen2016geometric}.
In Machine Learning we typically meet the problems with hidden affine structure and
small effective dimension (SVM) that allow us to use different smoothing techniques \cite{allen2016optimal}.
Description of one of these techniques (Nesterov's smoothing technique) one can find in this survey.
The other popular technique is based on averaging of the function around
the small ball with the center at the point in consideration \cite{duchi2012randomized}.
A huge amount of data since applications lead to composite optimization problems with non smooth composite (LASSO).
For this class of problems accelerated (fast) gradient methods are typically applied
\cite{beck2009fast}, \cite{nesterov2013gradient}, \cite{lan2016gradient}.
This approach (composite optimization) have been recently expanded for more general class of problems \cite{tyurin2017fast}.
In different Image Processing applications one can find a lot of non-smooth problems formulations with saddle-point structure.
That is the goal function has Legendre representation.
In this case one can apply special versions of accelerated (primal-dual) methods
\cite{chambolle2011first-order}, \cite{chen2014optimal}, \cite{lan2016accelerated}.
Universal Mirror Prox method described above demonstrates the alternative approach
which can be applied in rather general context.
Unfortunately, the most of these tricks have proven to be beyond the scope of this survey.
But we include in the survey the description of the Universal Accelerated Gradient
Descent algorithm
\cite{tyurin2017fast}
which in the general case can also be applied to a wide variety of problems.

Another important direction in Non-smooth Convex Optimization is
huge-scale optimization for sparse problems \cite{nesterov2014subgradient}.
The basic idea that reduce huge dimension to non-smoothness is as follows:
$$\iprod{\ba_k}{\x} - b_k \le 0, \quad k = 1,\ldots m, \quad m\gg 1 $$
is equivalent to the single non-smooth constraint:
$$\max_{k = 1,\ldots m}\left\{\iprod{\ba_k}{\x} - b_k \right\} \le 0.$$
We demonstrated this idea above on Truss Topology Design example.

One should note that we concentrate in this survey only on deterministic convex optimization problems,
but the most beautiful things in non smooth optimization is that stochasticity
\cite{nemirovsky1983problem}, \cite{duchi2016introductory}, \cite{juditsky2011firstI}, \cite{juditsky2011firstII}
and online context
\cite{hazan2016introduction}
in general doesn't change (up to a logarithmic factor in the strongly convex case)
anything in complexity estimates.
As an example, of stochastic (randomized) approach one can mentioned the work
\cite{anikin2015efficient}
where one can find reformulation of Google problem as non smooth convex optimization problem.
Special randomized Mirror Descent algorithm allows to solve this problem
almost independently on the number of vertexes.

Finally, let's note that in the decentralized distributed non smooth (stochastic) convex optimization
for the last few years there appear optimal methods
\cite{lan2017communication}, \cite{uribe2017optimal}, \cite{bubeckoptimal}.
\begin{acknowledgement}
The article was supported in its major parts by the grant 18-29-03071 mk from Russian Foundation for
Basic Research. E. Nurminski acknowledges the partial support from the project
1.7658.2017/6.7 of Ministry of Science and Higher Professional Education in Section \ref{exap:trans}.
The work of A. Gasnikov, P. Dvurechensky and F. Stonyakin in Section \ref{black-box} was partially supported by Russian
Foundation for Basic Research grant 18-31-20005 mol\_a\_ved.  The work of F. Stonyakin in Subsection \ref{S:CNSOF}
was supported by Russian Science Foundation grant 18-71-00048. 
\end{acknowledgement}
\bibliography{fullrefs}
\end{document}